\documentclass[11pt]{amsart}

\usepackage{multirow,bigdelim}
\usepackage{amsfonts}
\usepackage{dsfont}
\usepackage{amsmath}
\usepackage{mathrsfs}
\usepackage{todonotes}
\usepackage{hyperref}
\usepackage{multimedia}
\usepackage{graphicx}
\usepackage{indentfirst}
\usepackage{bm}
\usepackage{amsthm}
\usepackage{mathrsfs}
\usepackage{latexsym}
\usepackage{amsmath}
\usepackage{amssymb}
\usepackage{microtype}
\usepackage{relsize}
\usepackage{hyperref}
\usepackage{stmaryrd}

\newtheorem{thm}{Theorem}[section]

\newtheorem{lem}[thm]{Lemma}
\newtheorem{prop}[thm]{Proposition}
\newtheorem{defn}[thm]{Definition}

\theoremstyle{definition}
\newtheorem{rem}[thm]{Remark}

\newtheorem{qn}[thm]{Question}

\numberwithin{equation}{section}

\def\m{\mathcal}

\def\ind{\text{ind}}
\def\det{\text{det}}
\def\dim{\text{dim }}
\def\Ker{\text{Ker }}
\def\Ran{\text{Ran }}
\def\diag{\text{Diag }}

\begin{document}
\title[Essentially normal Cowen-Douglas Operators]{ Similarity Invariants of Essentially normal Cowen-Douglas Operators and Chern Polynomials\\
\vspace{0.7cm}
{\Small In Memory of R. G. Douglas}}

%    Information for first author
\author{Chunlan Jiang,  Kui Ji and Jinsong Wu}
%    Current address
\curraddr[C. Jiang and K. Ji]{Department of Mathematics, Hebei Normal University, Shijiazhuang, Hebei 050016, China} 

\curraddr[J. Wu]{Institute for Advanced Study in Mathematics, Harbin Institute of Technology, Heilongjiang 150001, China}

\email[C. Jiang]{cljiang@hebtu.edu.cn}
\email[K. Ji]{jikui@hebtu.edu.cn, jikuikui@163.com}
\email[J. Wu]{wjs@hit.edu.cn}

\thanks{
The first and the second authors were supported by National Natural Science Foundation of China (Grant No. 11831006 and 11922108). 
The third author was supported by National Natural Science Foundation of China (Grant No. 11771413).
}



\begin{abstract} 
In this paper, we systematically study a class of essentially normal operators by using the geometry method from the Cowen-Douglas theory and prove a Brown-Douglas-Fillmore theorem in the Cowen-Douglas theory. 
More precisely, the Chern polynomials and the second fundamental forms are the similarity invariants (in the sense of Herrero) of this class of essentially normal operators.
\end{abstract}

\maketitle

\section{Introduction}
People have long sought to establish connections between operator theory on the one hand, and geometry and algebraic topology on the other.
The Brown-Douglas-Fillmore theory \cite{BDF77} and the Cowen-Douglas operator theory \cite{CowDou78} together provide a powerful suite of tools for building such connections.
The Brown-Douglas-Fillmore theorem states that any two essentially normal operators are unitarily equivalent up to a compact perturbation (which is called essentially unitary equivalence) if and only if their essential spectrums and Fredholm indexes are the same.
In 1991, Berg and Davidson \cite{BerDav91} give a constructive proof of the Brown-Douglas-Fillmore theorem.
The Brown-Douglas-Fillmore theorem shows that the essentially unitary invaritants of an essentially normal operators are essential spectrums and Fredholm indexes.
It is highly nontrivial to obtain unitary (similarity) invariants of operators.
An alternative way is to study the closure of unitary (similarity) orbits.
In general, the closure of unitary orbit of a Hilbert space operator was completely determined by Hadwin \cite{Had77}.
The closure of similarity orbit of a Hilbert space operator was ``essentially" solved by Apostol, Herrero and Voiculescu (\cite{AHV82}, Chapter 9 of \cite{AFHV84}).
The similarity orbit of a normal operator was characterized by Fialkow \cite{Fia75}.
In 1986, Herrero \cite{Her86} introduce ($\m{U}+\m{K}$)-orbits between unitary orbits and similarity orbits in which the unitary operator is replaced by an invertible operator of the form ``unitary operator + compact operator".
Herrero \cite{Her86} ask if we can find a simple characterization of the closure of $(\m{U}+\m{K})$-orbit of a Hilbert space operator.
In 1993, Guinand and Marcoux \cite{GuiMar93a,GuiMar93b} completely characterized the the closure of ($\m{U}+\m{K}$)-orbits of normal operators, compact operators and certain weighted shifts respectively.
In 1998, the first author and Wang \cite{JiaWan98} completely characterize the closure of ($\m{U}+\m{K}$)-orbits of certain essentially normal operators.
Inspired by Herrero's ($\m{U}+\m{K}$)-equivalence and the Brown-Douglas-Fillmore theorem, one could ask
\begin{qn}\label{qn2}
What are the ($\m{U}+\m{K}$)-invariants for essentially normal operators?
\end{qn}
We have not found any (partial) results to this question in any literature.
Moreover, it is challenging to characterize the ($\m{U+K}$)-orbits or invariants of arbitrary operators. 

The Cowen-Douglas theory introduced geometry operators for a given open subset $\Omega\subset\mathbb{C}$ and a natural number $n$. 
These operators are called Cowen-Douglas operators over $\Omega$ with index $n$. (The set of these operators is denoted by $\m{B}_n(\Omega)$.)
Cowen and Douglas showed that each operator in $\m{B}_n(\Omega)$ gives rise a rank $n$ Hermitian holomorphic vector bundle over $\Omega$.
The curvature of the vector bundle defines the curvature of the operator.
It is an elegant and deep result that the unitary invariants of Cowen-Douglas operators are their curvatures and their covariant partial derivatives.
Based on their results, it is natural to ask what the similarity invariants of Cowen-Douglas operators are.
Cowen and Douglas \cite{CowDou78} conjectured that the Cowen-Douglas operators over the open unit disk with index one are similar to each other if and only if the limit of the quotient of the two curvatures is one.
The conjecture turns out to be false by Clark and Misra \cite{ClaMis83, ClaMis85}.
The failure of the conjecture leads mathematicians to find similarity invariants for Cowen-Douglas operators with additional structures.
In 2013, Douglas, Kwon and Treil \cite{DKT13} studied the similarity of hypercontractions and backward Bergman shifts. 
Hou, Ji and Kown \cite{HJK} showed that the geometry invariant is good for similarity invariants.
In 2017, the first two authors, Keshari and Misra introduced and studied a class of Cowen-Douglas operators with flag structures, denoted by $\m{FB}_n(\Omega)$.
Later, the first two authors and Misra \cite{JJM17} introduced and studied a class Cowen-Douglas operators with a little more structures.
In 2019, the first two authors and Keshari \cite{JJK19} found that the curvatures and the second fundamental forms can completely characterize the similarity invariant for a norm dense class of Cowen-Douglas operators denoted by $\m{CFB}_n(\Omega)$.

Thanks to the Brown-Douglas-Fillmore theory and the Cowen-Douglas theory, problems in geometry and algebraic topology can be approached using techniques from operator theory.
Inspired by this, Douglas, Tang and Yu \cite{DTY16} used operator-theoritic methods to prove an analytic Grothendieck-Riemann-Roch theorem for certain spaces with singularities.
In this paper, we proceed in the reverse direction, using methods from complex geometry to find a complete set of invariants for the similarity problem for essentially normal operators.
Besides the curvature, the Chern polynomial is also important in geometry.
We partially answer Questions \ref{qn2} for certain essentially normal Cowen-Douglas operators. (The set of these operators is denoted by $\m{NCFB}_n(\Omega)$.)
We also show that the set $\m{NCFB}_n(\Omega)$ contains enough nontrivial operators (see Theorem \ref{thm:ukorbit}).
It is surprising for us to find that the Chern polynomials and the second fundamental forms can completely characterize the ($\m{U}+\m{K}$)-invariants for $\m{NCFB}_n(\Omega)$ (see Theorem \ref{thm:similarity} and \ref{thm:unitary}).
This result is a version of the Brown-Douglas-Fillmore theorem in the Cowen-Douglas theory.
It is important to notice that the Chern polynomials can not be taken as similarity invariants for $\m{CFB}_n(\Omega)$.
Our results show the Chern polynomial can significantly simplify the ($\m{U}+\m{K}$)-invariants for $\m{NCFB}_n(\Omega)$.
It is also worth to point out that the previous study on $(\m{U+K})$-orbits requires that the spectrum of the operator has to be simply connected.
However, our results do not have assumption on the spectrum.
In a word, our approach offers a natural complement to that of Douglas, Tang and Yu \cite{DTY16}.

The paper is organized as follows. 
In Section 2 we recall the definition and some properties of Cowen-Douglas operators and the subclasses $\m{FB}_n(\Omega), \m{CFB}_n(\Omega)$ of Cowen-Douglas operators.
In Section 3 we introduce the class $\m{NCFB}_n(\Omega)$ of essentially normal Cowen-Douglas operators and study the structures of the operators and present some examples.
In Section 4 we obtain completely the ($\m{U+K}$)-invariants of $\m{NCFB}_n(\Omega)$.
In Section 5 we obtain a complete set of unitary invariants of in $\m{NCFB}_n(\Omega)$.
In Section 6 we suggest some further research on ($\m{U+K}$)-orbits and ($\m{U+K}$)-homogeneous.

\section{Preliminaries}

Throughout this paper, we denote $\m{H}$ by an infinite-dimensional separable Hilbert space, $\m{L}(\m{H})$ the algebra of all bounded operators on $\m{H}$, $\m{K}(\m{H})$ the algebra of all compact operators on $\m{H}$, and $\m{A}(\m{H})$ the Calkin algebra $\m{L}(\m{H})/\m{K}(\m{H})$.
Let $\m{Q}: \m{L}(\m{H})\to \m{A}(\m{H})$ be the quotient map onto the Calkin algebra.
An operator $T$ on $\m{H}$ is essentially normal if $\m{Q}(T)$ is normal, i.e. $T^*T-TT^*\in \m{K}(\m{H})$.

We are interested in the following equivalences of operators: for any $T_1, T_2\in \m{L}(\m{H})$, we say
\begin{enumerate}
\item $T_1$ is unitarily equivalent to $T_2$ if there is a unitary operator $U\in \m{L}(\m{H})$ such that $T_1U=UT_2$, denoted by $T_1\sim_{u} T_2$;
\item $T_1$ is ($\m{U}+\m{K}$)-equivalent to $T_2$ if there are a unitary operator $U\in \m{L}(\m{H})$ and a compact operator $K\in \m{K}(\m{H})$ such that $U+K$ is invertible and $T_1(U+K)=(U+K)T_2$, denoted by $T_1\sim_{\m{U+K}} T_2$;
\item $T_1$ is similar to $T_2$ if there is an invertible operator $S\in \m{L}(\m{H})$ such that $T_1S=ST_2$, denoted by $T_1\sim_{s} T_2$;
\end{enumerate}

\begin{rem}
Suppose that $T_1, T_2\in\m{L}(\m{H})$ are essentially unitary equivalent.
Usually, we do not have $T_1\sim_{\m{U+K}} T_2$.
Suppose that $\{e_{n}\}_{n\geq 0}$ is an orthogonal normal basis of $\m{H}$ and $T_1$ is the shift on $\m{H}$, i.e. $T_1e_{n}=e_{n+1}, \ n\geq 0$ and $T_2$ is a Bergman shift such that $T_2e_n=\displaystyle \sqrt{\frac{n}{n+1}} e_{n+1}, \ n\geq 0$.
Then $T_1-T_2$ is a compact operator such that $(T_1-T_2)e_n=\displaystyle \frac{1}{n+1+\sqrt{n(n+1)}}e_{n+1}$ and $T_1$ is not similar to $T_2$.
In fact, if $T_1\sim_{s}T_2$, i.e. $T_1X=XT_2$ for some invertible operator $X$, we have that 
\begin{align*}
\langle X e_{n}, e_{n} \rangle=\displaystyle \sqrt{\frac{n+1}{n}}\langle XT_2 e_{n-1}, e_{n}\rangle=\displaystyle \sqrt{\frac{n+1}{n}}\langle  X e_{n-1}, T_1^* e_{n}\rangle=\displaystyle \sqrt{\frac{n+1}{n}}\langle X e_{n-1}, e_{n-1}\rangle,
\end{align*}
i.e. 
\begin{align*}
\lim_{n\to \infty}\langle X e_{n}, e_{n} \rangle= \lim_{n\to \infty}\sqrt{n+1}\langle Xe_0, e_0\rangle=\infty.
\end{align*}
We see that $X$ is not bounded, which leads a contradiction and $T_1$ is not similar to $T_2$.
\end{rem}

We now recall the setup for Cowen-Douglas operators (geometry operators) in the following.
Let $\m{H}$ be a separable infinite-dimensional Hilbert space and $\text{Gr}(n, \m{H})$ the set of all $n$-dimensional subspaces of $\m{H}$ for any $n\in\mathbb{N}$.
Suppose $\Omega\subset\mathbb{C}$ is an open bounded connected subset and $n\in \mathbb N$. 
A map $t: \Omega \rightarrow \text{Gr}(n,\m{H})$ is a holomorphic curve if there exist $n$ (point-wise linearly independent) holomorphic functions $ \gamma_1 ,\gamma_2 ,\cdots, \gamma_n$ on $\Omega$ taking values in $\m{H}$ such that $t(w)=\bigvee \{\gamma_1(w),\cdots,\gamma_n(w)\},$ $w\in \Omega.$ 
Each holomorphic curve $t:\Omega \rightarrow \text{Gr}(n,\m{H})$ gives rise to a rank $n$ Hermitian holomorphic vector bundle $E_t$ over $\Omega$, namely,
$$E_{t}=\left\{(x,w)\in \m{H}\times \Omega \mid x\in t(w)\right\} \quad \text{and} \quad \pi:E_t\to \Omega,\,\,\mbox{where} \quad \pi(x,w)=w.$$
The set $\{\gamma_1 \ldots, \gamma_n\}$ of holomorphic functions is a holomorphic frame for $t$, denoted by $\gamma$.
An operator $T$ acting on  $\m{H}$ is said to be a Cowen-Douglas operator with index $n$ associated with an open bounded subset $\Omega$, if  
\begin{enumerate}
\item $T-w$ is surjective, 
\item $\dim \Ker(T-w)=n$ for all $w\in \Omega$, 
\item $\bigvee\limits_{w\in \Omega} \Ker(T-w)=\mathcal{H}$. 
\end{enumerate}
The set of Cowen-Douglas operators is denoted by $\m{B}_n(\Omega)$. 

Cowen and Douglas\cite{CowDou78} showed that each operator $T$ in $\m{B}_n(\Omega)$ gives also rise to a rank $n$ Hermitian holomorphic vector bundle $E_T$ over $\Omega$, 
$$E_{T}=\{(x,w)\in \m{H}\times \Omega \mid x\in \Ker (T-w)\} \quad \text{and} \quad \pi:E_T\rightarrow \Omega \quad \text{where}\quad \pi(x,w)=w.$$

Two holomorphic curves $t$ and $\tilde{t}$ are said to be 
\begin{enumerate}
\item unitarily equivalent (denoted by $t\sim_u\tilde{t}$), if there exists a unitary operator $U \in \m{L}(\m{H})$ such that $U(w)t(w)=\tilde{t}(w),$  where $U(w):=U|_{E_t(w)}$ is the restriction of the unitary operator $U$ to the fiber $E_t(w)=\pi^{-1}(w)$. 
\item similar (denoted by $t\sim_s \tilde{t}$) if there exists an invertible operator $X\in \m{L}(\m{H})$ such that $X(w)t(w)={\tilde{t}(w)}$, where $X(w):= X|_{E_t(w)}$ is the restriction of $X$ to the fiber $E_t(w)$. 
 In this case, we say that the vector bundles $E_t$ and $E_{\tilde{t}}$ are similar.
 \end{enumerate}

For an open bounded connected subset $\Omega$ of $\mathbb{C}$, a Cowen-Douglas operator $T$ with index $n$  induces a non-constant holomorphic curve $t:\Omega \to \text{Gr}(n,\m{H})$, namely,  $t(w)=\Ker(T-w), w\in \Omega$ and hence a Hermitian holomorphic vector bundle $E_t$ (here vector bundle $E_t$ is same as $E_T$). 
Unitary and similarity invariants for the operator $T$ are obtained from the vector bundle $E_T$. 

The metric of the Hermitian holomorphic vector bundle $E_T$ with respect to a holomorphic frame $\gamma$ is given by  the $n\times n$ matrix
$$h_\gamma(w)=\left (\langle \gamma_j(w),\gamma_k(w)\rangle \right )_{k,j=1}^n,\quad  w\in \Omega.$$ 
Note that $\{\gamma_1(w), \ldots, \gamma_n(w)\}$ is linearly independent.
The matrix $h_\gamma(w)$ is invertible.
The curvature of the Hermitian holomorphic vector bundle $E_T$ is 
\begin{equation}\label{eq:curbundle}
\bar{\partial}\left (h_\gamma^{-1}(w){\partial}h_\gamma(w)\right )=-\frac{\partial}{\partial \overline{w}}\left (h_\gamma^{-1}(w)\frac{\partial}{\partial w}h_\gamma(w)\right)\, dw \wedge d\bar{w},
\end{equation}
where $\partial$ is the differential operator with respect to the complex variable $w$ and $\bar{\partial}$ is with respect to the conjugation $\bar{w}$.
The curvature $K_T$ of the Cowen-Douglas operator $T$ over $\Omega$ is given by the coefficient of Equation (\ref{eq:curbundle}):
\begin{align*}
K_T(w):= -\frac{\partial}{\partial \overline{w}}\left(h_\gamma^{-1}(w)\frac{\partial}{\partial w}
h_\gamma(w)\right).
\end{align*} 
The total Chern form of the holomorphic bundle $E_T$ is defined to be
\begin{align*}
C_{E_T}=\det\left(I+\frac{i}{2\pi}K_T\right).
\end{align*}
The Chern polynomial $C_{E_T}(w)$ of $E_T$ with respect to a complex variable $w$ is defined to be
$$C_{E_T}(w)=\det\left(I+w \frac{i}{2\pi}K_T\right).$$

Next, we will describe the structure of Cowen-Douglas operators in $\m{B}_n(\Omega)$.
\begin{thm}[{Upper triangular representation theorem}, \cite{JiaWan98}] 
Let $T\in\m{B}_n(\Omega)$.
Then there exists an orthogonal decomposition $\mathcal{H}=\mathcal{H}_1\oplus\mathcal{H}_2\oplus\cdots\oplus\mathcal{H}_n$ and operators $T_{1,1}, T_{2,2}, \cdots, T_{n,n}$ in $\m{B}_1(\Omega)$ such that $T$ has the following form 
\begin{equation} \label{1.1T}
T=\left ( \begin{smallmatrix}T_{1,1} & T_{1,2}& T_{1,3}& \cdots & T_{1,n}\\
0&T_{2,2}&T_{2,3}&\cdots&T_{2,n}\\
\vdots&\ddots&\ddots&\ddots&\vdots\\
0&\cdots&0&T_{n-1,n-1}&T_{n-1,n}\\
0&\cdots&\cdots&0&T_{n,n}
\end{smallmatrix}\right ).
\end{equation}
Let 
$\{\gamma_1,\gamma_2,\cdots,\gamma_{n}\}$ be a holomorphic frame
of  $E_T$ with ${\mathcal H}=
\bigvee\{\gamma_j(w),\;w\in \Omega,\;1\leq j \leq n\}, $  and   $t_j:\Omega \to
\mbox{Gr}(1,{\mathcal H}_j)$  be a holomorphic frame of $E_{T_{j,j}},\,1\leq j \leq n$. 
Then we can find certain relation between $\{\gamma_j\}^{n}_{j=1}$ and $\{t_j\}^{n}_{j=1}$ as the following equations:
\begin{equation}\label{1.1}
\begin{array}{llll}\gamma_1&=&t_1\\
                   \gamma_{2}&=&\phi_{1,2}(t_2)+t_{2}\\
                   \gamma_{3}&=&\phi_{1,3}(t_3)+\phi_{2,3}(t_3)+t_{3}\\
                  \cdots &=& \cdots \cdots\cdots\cdots\cdots\cdots\cdots\cdots\\
                    \gamma_{j}&=&\phi_{1,j}(t_j)+\cdots+\phi_{i,j}(t_j)+\cdots+t_{j}\\
                    \cdots &=& \cdots \cdots\cdots\cdots\cdots\cdots\cdots\cdots\cdots\cdots\cdots\\
                   \gamma_{n}&=&\phi_{1,n}(t_{n})+\cdots+\phi_{i,n}(t_{n})+\cdots+t_{n},\\
\end{array}
\end{equation}
where $\phi_{i,j}, (i,j=1,2,\cdots,n)$ are certain holomorphic bundle maps.
\end{thm}

In the rest of the paper, we always assume that a Cowen-Douglas operator $T\in \m{B}_n(\Omega)$ has a matrix form $(T_{k,j})_{k,j=1}^n$ with respect to an orthogonal decomposition of the Hilbert space such that the diagonals are in $\m{B}_1(\Omega)$.

Recall that the commutant $\{T\}^{\prime}$ of $T$ on $\m{H}$ is the set of operators in $\m{L}(\m{H})$ commuting with $T$.
The operator $T$ is strongly irreducible if there is no non trivial idempotent in $\{T\}^{\prime}.$

An operator $T$ in ${\mathcal L}(\mathcal H)$ is said to be strongly
irreducible, if there is no non-trivial idempotent operator in
$\{T\}^{\prime}$, where $\{T\}^{\prime}$ denotes
the commutant of $T$, i.e., $\{T\}^{\prime}=\{B{\in}{\mathcal L}({\mathcal H}): {TB=BT}\}$. 
It is known that for any $T\in \m{B}_1(\Omega)$, $T$ is strongly irreducible. 
Every Cowen-Douglas operator can be written as the direct sum of finitely many strongly irreducible Cowen-Douglas operators (see \cite{JiaWan98}, chapter 3).
By the results in \cite{JW1}, the similarity of Cowen-Douglas operators can be reduced to the similarity of strongly irreducible Cowen-Douglas operators.

Now, we recall some classes of Cowen-Douglas operators with additional structures.

\begin{defn} 
Let $\mathcal{FB}_n(\Omega)\subset \m{B}_n(\Omega)$ be the set of $T\in\m{B}_n(\Omega)$ on $\m{H} = \m{H}_1 \oplus \cdots \oplus \m{H}_{n},$ having the following form
$$T=\left ( \begin{matrix}
T_{1,1} & T_{1,2} &\cdots&T_{1,n}\\
&T_{2,2}&\cdots&T_{2,n} \\
&&\ddots&\vdots\\
&&&T_{n,n}\\
\end{matrix} \right )$$
such that 
\begin{enumerate}
\item $T_{j,j}\in \m{B}_1(\Omega)$ for $1\leq j\leq n$;
\item $T_{j,j}T_{j,j+1}=T_{j,j+1}T_{j+1, j+1}$ for any $1\leq j\leq n-1$, i.e. $T_{j,j+1}$ intertwines $T_{j,j}$ and $T_{j+1, j+1}$.
\end{enumerate}
\end{defn}

\begin{defn} 
Let $T_1, T_2\in\m{L}(\m{H})$.
Then the Rosenblum operators $\tau_{T_1,T_2}$ is the map from $\m{L}(\m{H})$ to $\m{L}(\m{H})$ defined by:
$$\tau_{T_1,T_2}(X)=T_{1}X-XT_{2}, \quad X\in\m{L}(\m{H}). $$ 
If $T_1=T_2=T$, we let $\delta_{T}=\tau_{T, T}.$
\end{defn}

\begin{defn}[\bf{Property (H)}]\label{ph}
Let $T_1, T_2\in \m{L}(\m{H})$. 
The ordered pair $(T_1,T_2)$ has the \textit{Property $(H)$} if the following condition holds:
if there exists $X\in \m{L}(\m{H})$ such that  
\begin{enumerate}
\item $T_{1}X=XT_{2}$,
\item $X=T_{1}Z-ZT_{2},\quad \text{for some}\; Z\in \m{L}(\m{H})$, 
\end{enumerate}
then $X=0$.
\end{defn}

\begin{defn}[Definition 2.7 in \cite{JJK19}]\label{CFB1}
Let $\mathcal{CFB}_n(\Omega)$ be the subset of operators $T\in \m{B}_n(\Omega)$ satisfies the following properties: 
\begin{enumerate}
\item $T$ can be written as an $n\times n$ upper-triangular matrix from with respect to a topological direct decomposition of the Hilbert space and  $\diag T\in \{T\}^{\prime}$,
 where $\diag T$ is the diagonal of $T$.
\item For $1\leq k\leq j\leq n$, $T_{k,j}=C_{k,j}T_{k,k+1}T_{k+1,k+2} \cdots T_{j-1,j}$ for some $C_{k,j}\in\{T_{k,k}\}^{\prime}$ on $\m{H}_k$;
\item The ordered pair $(T_{j,j}, T_{j+1,j+1}), 1\leq j\leq n-1$ satisfies the Property $(H)$, that is,  
$$\Ker\tau_{T_{j,j},T_{j+1, j+1}}\cap \Ran\tau_{T_{j,j},T_{j+1,j+1}}=\{0\}, \quad 1\leq j\leq n-1.$$
\end{enumerate}
\end{defn}

\begin{rem}
According to 1.20 in \cite{CowDou78}, there exists a holomorphic function $\phi_{k,j}$ such that $C_{k,j}=\phi_{k,j}(T_{k,k})$ in Definition \ref{CFB1}.
\end{rem}

\begin{lem} \label{SI}\cite{JJK19}
Let $T\in \mathcal{CFB}_n(\Omega)$. 
Then $T$ is strongly irreducible  if and only if  $T_{j,j+1}\neq 0$ for any $j=1,2\cdots, n-1$.
\end{lem}

In the end of this section, we recall some results for $\m{B}_1(\Omega)$ and some technique lemmas.
\begin{prop}[Proposition 4.12 in \cite{JJK19}]\label{u+k}
Let $T, \widetilde{T}\in \m{B}_1(\Omega)$.
Then $T\sim_{\m{U+K}}\widetilde{T}$ if and only if  there  exists $X\in \m{L}(\m{H})$  such that $\m{Q}(X)=\alpha \m{Q}(I)$, $0<\alpha<1$ and 
$$K_{T}(w)-K_{\widetilde{T}}(w)=\frac{\partial^2}{\partial w\partial\overline{w}} \Psi(w), \quad \Psi(w)=\ln\left(\frac{\|X(t(w))\|^2}{\|t(w)\|^2}+(1-\alpha^2)\right)$$
where $t$ is a non zero section of $E_T$.  
\end{prop}

Thoughout this paper, we denoted by $\sigma(T)$ the spectrum of an operator $T\in\m{L}(\m{H})$.
\begin{lem}\label{Hal}\cite{Hal} 
Let $X, T\in\m{L}(\mathcal {H})$. 
If $X \in \Ker\delta_{T} \cap\Ran\delta_{T}$, then $\sigma(X)=\{0\}$. 
\end{lem}

\begin{lem}\label{Hal1}
Let $X, T\in\m{L}(\mathcal {H})$. 
If $\m{Q}(X)\in \Ker\delta_{\m{Q}(T)} \cap \Ran\delta_{\m{Q}(T)}$, then $\sigma(\m{Q}(X))=\{0\}$. 
\end{lem}
\begin{proof}
It follows from Lemma \ref{Hal} directly.
\end{proof}

Let $T\in \mathcal{FB}_n(\Omega)$. For a $2\times 2$ block 
$\begin{pmatrix}
T_{i,i}&T_{i,i+1}\\
0&T_{i+1,i+1}\\
\end{pmatrix}$, if  $T_{i,i}T_{i,i+1}=T_{i,i+1}T_{i+1,i+1},$  the corresponding second fundamental form $\theta_{i,i+1}$, which is obtained by R. G. Douglas and G. Misra (see in \cite{DM}), is 
\begin{equation} \label{sf}
\theta_{i,i+1}(T)(w) =  \frac{ K_{T_{i,i}}(w) \,d\bar{w}} {\big(\frac{\|T_{i,i+1}(t_{i+1}(w))\|^2}{\|t_{i+1}(w)\|^2} - 
K_{T_{i,i}}(w)\big )^{1/2}}.
\end{equation}

Let  $T,\widetilde{T}$ have upper-triangular representation as in $(1.1)$ and assume that $T_{i,i}, T_{i+1,i+1}$ and $\tilde{T}_{i,i}, \tilde{T}_{i+1,i+1}$ have intertwines $T_{i,i+1}$ and $\tilde{T}_{i,i+1}$, respectively.  
If $K_{T_{i,i}}= K_{\widetilde{T}_{i,i}}$, then from $(1.3)$ it is easy to see that $$\theta_{i,i+1}(T)(w) =\theta_{i,i+1}(\tilde{T)}(w) \Leftrightarrow \frac{\|T_{i,i+1}(t_{i+1}(w))\|^2}{\|t_{i+1}(w)\|^2}=\frac{\|\widetilde{T}_{i,i+1}(\widetilde{t}_{i+1}(w))\|^2}{\|\widetilde{t}_{i+1}(w)\|^2}.$$

\section{Essentially normal Cowen-Douglas Operators}

In this section, we introduce certain essentially normal Cowen-Douglas operators and study their basic properties.

\begin{defn}\label{ncfb}
Let $\m{EN}$ be the set of all essentially normal operators on the Hilbert space $\m{H}$.
For an open subset $\Omega$ in $\mathbb{C}$ and $n\in\mathbb{N}$, we define $\m{NCFB}_n(\Omega)$ to be a subset of $\m{EN}\cap \m{CFB}_n(\Omega)$ subject to the following condition:
for any $T\in \m{NCFB}_n(\Omega)$, there exist idempotents $P_1, \ldots, P_n$ satisfying $\sum\limits_{j=1}^n P_j=I$, $P_kP_j=0$, $k\neq j$ and $\m{Q}(P_j)^*=\m{Q}(P_j)$ for $j=1, \ldots,n$ such that $T\in \m{CFB}_n(\Omega)$ with respect to the topological direct decomposition $\Ran P_1\dotplus \cdots \dotplus \Ran P_n$.
\end{defn}

\begin{rem}
In Definition \ref{CFB1}, the topological direct decomposition can be perturbed by an invertible operator to an orthogonal decomposition.
To study ($\m{U+K}$)-invariant, we require that the topological direct decomposition can perturbed by a ($\m{U+K}$)-invertible operator to an orthogonal decomposition.
The additional condition in Definition \ref{ncfb} is to meet the requirement above (See the following Proposition \ref{nc} for a proof).
\end{rem}

\begin{prop}\label{nc} 
Suppose that ${\mathcal H}={\mathcal H}_1 \dotplus {\mathcal H}_2\dotplus \cdots  \dotplus{\mathcal H}_n$ is a topological direct decomposition of $\mathcal{H}$ and there exist idempotents $P_1, P_2,\ldots, P_n$ such that $\sum\limits_{j=1}^nP_j=I, P_kP_j=0,\, k\neq j$ where $\Ran P_j={\mathcal H}_j$. 
If there exists a compact operator $K_j$ such that $P_j=\left(\begin{smallmatrix}I_{\mathcal{H}_j}&K_j\\ 0&0\end{smallmatrix}\right)$ with respect to the Hilbert space decomposition $\m{H}_j\oplus\m{H}_j^\perp$. 
Then there exist a unitary operator $U$ and a compact operator $K$ such that $U+K$ is invertible, $Q_j=(U+K)P_j(U+K)^{-1}$ is an (orthogonal) projection, and $I=\sum_{j=1}^n Q_j$. 
\end{prop}

\begin{proof}
We assume that $P_1$ has the matrix form with respect to the orthogonal decomposition $\mathcal{H}=\mathcal{H}_1\oplus \mathcal{H}_1^{\perp}$ as
$$P_1=\left(\begin{matrix}I&P^{(1)}_{1,2}\\
0&0\\ \end{matrix}\right)\begin{small}\begin{matrix}
\mathcal{H}_1\\
\mathcal{H}_1^{\perp}
\end{matrix}\end{small},$$
where $P_{1,2}^{(1)}$ is a compact operator.
Suppose $K_{1,2}^{(1)}=-P^{(1)}_{1,2}$, then we have that 
$$\left(\begin{matrix}I&P^{(1)}_{1,2}\\
0&0\\ \end{matrix}\right) \left(\begin{matrix} I &K^{(1)}_{1,2}\\
0&I \\ \end{matrix}\right)=\left(\begin{matrix} I &K^{(1)}_{1,2}\\
0&I \\ \end{matrix}\right)\left(\begin{matrix}I&0\\
0&0\\ \end{matrix}\right).$$
Set $U_1+K_1=\left(\begin{matrix} I &K^{(1)}_{1,2}\\
0&I \\ \end{matrix}\right)$ and $Q_1=(U_1+K_1)P_1(U_1+K_1)^{-1}$, then 
$$I=Q_1\oplus (U_1+K_1)\left(\sum\limits_{j=2}^n P_j\right)(U_1+K_1)^{-1}.$$

Since $(U_1+K_1)\left(\sum\limits_{j=2}^nP_j\right)(U_1+K_1)^{-1}$ is a decomposition of $I_{\Ran(I-Q_1)}$, repeating the argument above for $(U_1+K_1)P_2 (U_1+K_1)^{-1}$, there exists $U_2+K_2$ such that 
$$(I\oplus (U_2+K_2))(0\oplus (U_2+K_2)P_2(U_2+K_2)^{-1})(I\oplus (U_2+K_2)^{-1})$$
is an orthogonal projection (denoted by $Q_2$).  
Then we have that 
\begin{eqnarray*}
I&=&Q_1\oplus Q_2\oplus (I\oplus (U_2+K_2)) (U_1+K_1)\left(\sum\limits_{i=3}^nP_i\right)(U_1+K_1)^{-1}(I\oplus (U_2+K_2)^{-1})\\
 &=& (I\oplus (U_2+K_2)) (U_1+K_1)\left(\sum\limits_{i=1}^nP_i\right)(U_1+K_1)^{-1}(I\oplus (U_2+K_2)^{-1})\end{eqnarray*}

By an inductive proof, we can find unitary operators $U_j$ and  compact operators $K_j, j=1,2\cdots, n$ such that 
$$Q_1\oplus Q_2\oplus\cdots Q_n=(I\oplus I\oplus \cdots(U_n+K_n))\cdots (U_1+K_1)\left(\sum\limits_{j=1}^n P_j\right)(U_1+K_1)^{-1}\cdots(I\oplus I\oplus \cdots(U_n+K_n)^{-1}).$$
This finishes the proof of this lemma. 

\end{proof}

For any topological direct decomposition of $\mathcal{H}$,  ${\mathcal H}={\mathcal H}_1 \dotplus {\mathcal H}_2\dotplus \cdots  \dotplus{\mathcal H}_n$, if there exist $n$ idempotents $P_1, P_2,\ldots, P_n$ such that $\sum\limits_{j=1}^nP_j=I, P_kP_j=0,\quad  k\neq j$ and $\Ran P_j={\mathcal H}_j$.  
Then we can find a unitary operator $U$ and a compact operator $K$  such that  $U+K$ is invertible and  $\left\{Q_j\right\}^n_{j=1}=\left\{(U+K)P_j(U+K)^{-1}\right\}^n_{j=1}$ be a set of orthogonal projections with  $Q_kQ_j=0,\,i\neq j$. 
Furthermore, 
$${\mathcal H}=(U+K){\mathcal H}_1\oplus (U+K){\mathcal H}_2 \oplus \cdots  \oplus (U+K){\mathcal H}_n,$$
where $(U+K){\mathcal H}_j= \Ran Q_j.$

\subsection{Structure}
We characterize the diagonals and off-diagonals of the operators in $\mathcal{NCFB}_n(\Omega)$ in the following theorem:

\begin{thm}\label{newl1}
Let $T\in \mathcal{NCFB}_n(\Omega)$.
Then we have that 
\begin{enumerate}
\item $T_{k,k}, 1\leq k\leq n$ are essentially normal; 
\item $T_{j,k}$ are compact for $1\leq j<k\leq n$.
\end{enumerate}
\end{thm}
\begin{proof} 
By the assumption that $T$ is essentially normal, we have $\m{Q}(T)\m{Q}(T^*)=\m{Q}(T^*)\m{Q}(T)$. 
Expanding it in the matrix form, we obtain that
\begin{equation}\label{compact} 
\begin{aligned}
&\left ( \begin{smallmatrix}
\m{Q}(T_{1,1}) & \m{Q}(T_{1,2}) &\cdots&\m{Q}(T_{1,n})\\
&\m{Q}(T_{2,2})&\cdots&\m{Q}(T_{2,n}) \\
&&\ddots&\vdots\\
&&&\m{Q}(T_{n,n})\\
\end{smallmatrix} \right )
\left ( \begin{smallmatrix}
\m{Q}(T^*_{1,1}) & &&\\
 \m{Q}(T^*_{1,2})&\m{Q}(T^*_{2,2})&& \\
\vdots&\vdots&\ddots& \\
\m{Q}(T^*_{1,n})&\m{Q}(T^*_{2,n})&\cdots&\m{Q}(T^*_{n,n})\\
\end{smallmatrix} \right )\\
=&\left ( \begin{smallmatrix}
\m{Q}(T^*_{1,1}) & &&\\
 \m{Q}(T^*_{1,2})&\m{Q}(T^*_{2,2})&& \\
\vdots&\vdots&\ddots& \\
\m{Q}(T^*_{1,n})&\m{Q}(T^*_{2,n})&\cdots&\m{Q}(T^*_{n,n})\\
\end{smallmatrix} \right )\left ( \begin{smallmatrix}
\m{Q}(T_{1,1}) & \m{Q}(T_{1,2}) &\cdots&\m{Q}(T_{1,n})\\
&\m{Q}(T_{2,2})&\cdots&\m{Q}(T_{2,n}) \\
&&\ddots&\vdots\\
&&&\m{Q}(T_{n,n})\\
\end{smallmatrix} \right ).
\end{aligned}
\end{equation} 
 Recall that $T\in \m{CFB}_n(\Omega)$, we see that 
 $$T_{k,k}T_{k,j}=T_{k,j}T_{j,j}, \quad 1\leq k<j\leq n,$$ and 
$$\m{Q}(T_{k,k})\m{Q}(T_{k,j})=\m{Q}(T_{k,j})\m{Q}(T_{j,j}), \quad  1\leq k<j\leq n.$$ 
Without loss of generality, we can assume that 
$$T_{k,j}=T_{k,k+1}T_{k+1,k+2}\cdots T_{j-1,j}, 1\leq k<j\leq n.$$
By Equation (\ref{compact}), we have that $$\m{Q}(T_{1,n}^*)\m{Q}(T_{1,1})=\m{Q}(T_{n,n})\m{Q}(T_{1,n}^*),$$ and 
$$\m{Q}(T_{1,n}^*)\m{Q}(T_{1,2})+\m{Q}(T^*_{2,n})\m{Q}(T_{2,2})=\m{Q}(T_{n,n})\m{Q}(T_{2,n}^*).$$
Thus, we have that 
\begin{align*}
\m{Q}(T_{1,n}^*)\m{Q}(T_{1,2})\m{Q}(T_{2,n})+\m{Q}(T^*_{2,n})\m{Q}(T_{2,2})\m{Q}(T_{2,n})&=\m{Q}(T_{n,n})\m{Q}(T_{2,n}^*)\m{Q}(T_{2,n}),\\
\m{Q}(T_{1,n}^*)\m{Q}(T_{1,n})+\m{Q}(T^*_{2,n})\m{Q}(T_{2,n})\m{Q}(T_{n,n})&=\m{Q}(T_{n,n})\m{Q}(T_{2,n}^*)\m{Q}(T_{2,n}).\\
\end{align*}
This indicates that 
$$\m{Q}(T_{1,n}^*)\m{Q}(T_{1,n})=\m{Q}(T_{n,n})\m{Q}(T_{2,n}^*)\m{Q}(T_{2,n})-\m{Q}(T^*_{2,n})\m{Q}(T_{2,n})\m{Q}(T_{n,n}).$$
On the other hand,  since $\m{Q}(T_{1,n}^*)\m{Q}(T_{1,1})=\m{Q}(T_{n,n})\m{Q}(T_{1,n}^*)$, we have that 
$$\m{Q}(T_{1,n}^*)\m{Q}(T_{1,n})\m{Q}(T_{n,n})=\m{Q}(T_{n,n})\m{Q}(T_{1,n}^*)\m{Q}(T_{1,n}).$$
That means $\m{Q}(T_{1,n}^*)\m{Q}(T_{1,n})\in \Ker\delta_{\m{Q}(T_{n,n})} \cap \Ran\delta_{\m{Q}(T_{n,n})}.$
By Lemma \ref{Hal1}, we have that 
$$\sigma\left(\m{Q}(T_{1,n}^*)\m{Q}(T_{1,n})\right)=\{0\}.$$  
Since $\m{Q}(T_{1,n}^*)\m{Q}(T_{1,n})$ is positive, it follows that $\m{Q}(T_{1,n}^*)\m{Q}(T_{1,n})=0$.
Hence $\m{Q}(T_{1,n})=0$, i.e. $T_{1,n}$ is a compact operator. 
 
 Now Equation (\ref{compact}) becomes:
\begin{equation}\label{compact0} 
\begin{aligned}
&\left ( \begin{smallmatrix}
\m{Q}(T_{1,1}) & \m{Q}(T_{1,2}) &\cdots&0\\
&\m{Q}(T_{2,2})&\cdots&\m{Q}(T_{2,n}) \\
&&\ddots&\vdots\\
&&&\m{Q}(T_{n,n})\\
\end{smallmatrix} \right )
\left ( \begin{smallmatrix}
\m{Q}(T^*_{1,1}) & &&\\
 \m{Q}(T^*_{1,2})&\m{Q}(T^*_{2,2})&& \\
\vdots&\vdots&\ddots& \\
0&\m{Q}(T^*_{2,n})&\cdots&\m{Q}(T^*_{n,n})\\
\end{smallmatrix} \right )\\
=&\left ( \begin{smallmatrix}
\m{Q}(T^*_{1,1}) & &&\\
 \m{Q}(T^*_{1,2})&\m{Q}(T^*_{2,2})&& \\
\vdots&\vdots&\ddots& \\
0&\m{Q}(T^*_{2,n})&\cdots&\m{Q}(T^*_{n,n})\\
\end{smallmatrix} \right )\left ( \begin{smallmatrix}
\m{Q}(T_{1,1}) & \m{Q}(T_{1,2}) &\cdots&0\\
&\m{Q}(T_{2,2})&\cdots&\m{Q}(T_{2,n}) \\
&&\ddots&\vdots\\
&&&\m{Q}(T_{n,n})\\
\end{smallmatrix} \right ).
\end{aligned}
\end{equation} 
By comparing the $(n,2)$, $(n,3)$-entries in Equation (\ref{compact0}), we have that 
$$\m{Q}(T_{2,n}^*)\m{Q}(T_{2,2})=\m{Q}(T_{n,n})\m{Q}(T_{2,n}^*),$$ and 
$$\m{Q}(T_{2,n}^*)\m{Q}(T_{2,3})+\m{Q}(T^*_{3,n})\m{Q}(T_{3,3})=\m{Q}(T_{n,n})\m{Q}(T_{3,n}^*).$$
 Repeating the same argument above, we also have that $\m{Q}(T_{2,n})=0$, i.e. $T_{2,n}$ is a compact operator.  Similarly, we can obtain that  $\m{Q}(T_{k,n})=0, k<n$. 
It follows that Equation (\ref{compact0}) becomes:
\begin{eqnarray*}\label{compact1}
 &&\left ( \begin{smallmatrix}
\m{Q}(T_{1,1}) & \m{Q}(T_{1,2}) &\cdots&\m{Q}(T_{1,n-1})&0\\
&\m{Q}(T_{2,2})&\cdots&\m{Q}(T_{2,n-1})&0 \\
&&\ddots&\ddots&\vdots\\
&&&\m{Q}(T_{n-1,n-1})&0\\
&&&&\m{Q}(T_{n,n})\\
\end{smallmatrix} \right )\left ( \begin{smallmatrix}
\m{Q}(T^*_{1,1}) && &&\\
 \m{Q}(T^*_{1,2})&\m{Q}(T^*_{2,2})&&& \\
\vdots&\vdots&\ddots& \\
\m{Q}(T^*_{1,n-1})&\cdots&\cdots&\m{Q}(T^*_{n-1,n-1})&\\
0&0&0&0&\m{Q}(T^*_{n,n})\\
\end{smallmatrix} \right )\\
&=&\left ( \begin{smallmatrix}
\m{Q}(T^*_{1,1}) && &&\\
 \m{Q}(T^*_{1,2})&\m{Q}(T^*_{2,2})&&& \\
\vdots&\vdots&\ddots& \\
\m{Q}(T^*_{1,n-1})&\cdots&\cdots&\m{Q}(T^*_{n-1,n-1})&\\
0&0&0&0&\m{Q}(T^*_{n,n})\\
\end{smallmatrix} \right )\left ( \begin{smallmatrix}
\m{Q}(T_{1,1}) & \m{Q}(T_{1,2}) &\cdots&\m{Q}(T_{1,n-1})&0\\
&\m{Q}(T_{2,2})&\cdots&\m{Q}(T_{2,n-1})&0 \\
&&\ddots&\ddots&\vdots\\
&&&\m{Q}(T_{n-1,n-1})&0\\
&&&&\m{Q}(T_{n,n})\\
\end{smallmatrix} \right ).\end{eqnarray*}
Then we could focus on the $(n-1)\times (n-1)$-matrix forms:
\begin{equation}\label{compact2} 
\begin{aligned}
&\left ( \begin{smallmatrix}
\m{Q}(T_{1,1}) & \m{Q}(T_{1,2}) &\cdots&\m{Q}(T_{1,n-1})\\
&\m{Q}(T_{2,2})&\cdots&\m{Q}(T_{2,n-1}) \\
&&\ddots&\vdots\\
&&&\m{Q}(T_{n-1,n-1})\\
\end{smallmatrix} \right )
\left ( \begin{smallmatrix}
\m{Q}(T^*_{1,1}) & &&\\
 \m{Q}(T^*_{1,2})&\m{Q}(T^*_{2,2})&& \\
\vdots&\vdots&\ddots& \\
\m{Q}(T^*_{1,n-1})&\m{Q}(T^*_{2,n-1})&\cdots&\m{Q}(T^*_{n-1,n-1})\\
\end{smallmatrix} \right )\\
=&\left ( \begin{smallmatrix}
\m{Q}(T^*_{1,1}) & &&\\
 \m{Q}(T^*_{1,2})&\m{Q}(T^*_{2,2})&& \\
\vdots&\vdots&\ddots& \\
\m{Q}(T^*_{1,n-1})&\m{Q}(T^*_{2,n-1})&\cdots&\m{Q}(T^*_{n-1,n-1})\\
\end{smallmatrix} \right )\left ( \begin{smallmatrix}
\m{Q}(T_{1,1}) & \m{Q}(T_{1,2}) &\cdots&\m{Q}(T_{1,n-1})\\
&\m{Q}(T_{2,2})&\cdots&\m{Q}(T_{2,n-1}) \\
&&\ddots&\vdots\\
&&&\m{Q}(T_{n-1,n-1})\\
\end{smallmatrix} \right ).
\end{aligned}
\end{equation} 
Notice that when $n=2$,  the conclusions holds by the same argument. Thus, by the inductive proof, we have that $\m{Q}(T_{k,j})=0$, for any $j>k$.  This finishes the proof of statement (1).
Furthermore, we have that 
$$ \m{Q}(T)=\left ( \begin{smallmatrix}
\m{Q}(T_{1,1}) & &&\\
&\m{Q}(T_{2,2})&& \\
&&\ddots&\\
&&&\m{Q}(T_{n,n})\\
\end{smallmatrix} \right ).$$
Thus, $\m{Q}(T)\m{Q}(T^*)=\m{Q}(T^*)\m{Q}(T)$ implies that $\m{Q}(T_{k,k})\m{Q}(T^*_{k,k})=\m{Q}(T^*_{k,k})\m{Q}(T_{k,k}), 1\leq k\leq n$.  This finishes the proof of statement  (2). 

\end{proof}

\subsection{Examples}
In this subsection, we show that there exists a strongly irreducible operator in $\m{NCFB}_n(\Omega)$ for any simply connected open subset $\Omega\subset\mathbb{C}$.

Let $d\mu$ denote the normalized Lebesgue area measure on the unit disk $\mathbb{D}$.
For any $\lambda>-1$, the weighted Bergman space $\mathbb{A}_{\lambda}^2$ is the space of analytic functions on $\mathbb{D}$ which are square-integrable with respect to the measure $d\mu_\lambda=(\lambda+1)(1-|z|^2)^\lambda d\mu(z)$ (See \cite{JZ} for more details on weighted Bergman spaces).
Let $K(z,w)=\displaystyle \frac{1}{(1-zw)^\lambda}$ be its positive definite kernel.

\begin{prop} 
Let $T$ be a homogenous operator in $\m{B}_1(\mathbb{D})$. 
Then $T\in \m{NCFB}_1(\mathbb{D})$.
\end{prop}

\begin{proof} 
Let $T\in B_1(\mathbb{D})$ be a homogenous operator.  
Then there exists $\lambda>0$ and shifts $M_z$ on $\mathbb{A}_{\lambda}^2$ such that $T\sim_u M^*_z$. 
By a direct computation,  there exist rank-one operators $K_0$ and $K_1$ such that $M_zM^*_z=I-K_0, M^*_zM_z=I-K_1.$  
That means $M^*_z$ is essentially normal. 
Since $T\sim_u M^*_z$, then $T$ is also essentially normal.
\end{proof}

 \begin{defn}
 Let $T_1,T_2\in \m{B}_1(\Omega)$. 
 We call $T_1\prec T_2$ if  $\Ker\tau_{T_1,T_2}\neq \{0\}, \text{ and }\Ker\tau_{T_2,T_1}=\{0\}.$ 
\end{defn}

\begin{prop}\cite{TGJ}
The relation ``$\prec$'' is a strictly partially ordered relation for the operators in $\m{B}_1(\Omega)$.
\end{prop}

\begin{lem}\label{homo1} \cite{JJK19}
Let $T_j, j=1,2$ be the adjoint of the multiplication operator $M_z^{(\lambda_j)}$ acting on $\mathbb{A}_{\lambda_j}^2$.  
Then $T_1\prec T_2$ if and only if $\lambda_1<\lambda_2.$ 

\end{lem}

\begin{lem}\label{homo2}\cite{JJK19}
Let $T_k= M_z^{(\lambda_k)*}, k=1,2$ be the adjoint of the multiplication operator $M_z^{(\lambda_k)}$ on $\mathbb{A}_{\lambda_k}^2$.  
If $\lambda_2-\lambda_1<2$, then $(T_1,T_2)$ satisfies Property $(H)$.
\end{lem}

\begin{proof} 
Suppose that there exists an non-zero operator $Y$ such that $T_1Y=YT_2$ and $Y=T_1X-XT_2$.  Notice that 
$$T_1=\begin{pmatrix}
0& \sqrt{\frac{1}{\lambda_1}} & 0&\cdots&0&\cdots\\
&0&\sqrt{\frac{2}{\lambda_1+1}} &\cdots&0 &\cdots\\
&&\ddots&\ddots&\vdots&\cdots\\
&&&0&\sqrt{\frac{k}{k+\lambda_1-1}}&\cdots\\
&&&&0&\ddots\\
&&&&&\ddots\\
\end{pmatrix}, T_2=\begin{pmatrix}
0& \sqrt{\frac{1}{\lambda_2}} & 0&\cdots&0&\cdots\\
&0&\sqrt{\frac{2}{\lambda_2+1}} &\cdots&0 &\cdots\\
&&\ddots&\ddots&\vdots&\cdots\\
&&&0&\sqrt{\frac{k}{k+\lambda_2-1}}&\cdots\\
&&&&0&\ddots\\
&&&&&\ddots\\
\end{pmatrix}. $$

By Lemma \ref{unhappy} and Lemma 4.3 in \cite{JJM17}, we know that $\lim\limits_{k\rightarrow \infty} x_{k,k+1}=\infty$ in the case of $\lambda_2-\lambda_1<2$.

\end{proof}

\begin{thm}\label{example}
Suppose that $\Omega$ is simply connected.
Then there exists $T\in \mathcal{NCFB}_n(\Omega)$ such that
 \begin{enumerate}
\item $T$ is strongly irreducible;
\item $\sigma(T)=\overline{\Omega}$, $\rho_F(T)\cap \sigma(T)=\Omega$;
\item  $\ind(w-T)=n, w\in \Omega$,
\end{enumerate}
where $\rho_{F}(T)$ denotes the Fredholm domain of $T$ and $\ind (T)$ is the Fredholm index of $T$.
In particular, $\m{NCFB}_n(\Omega)\neq \emptyset$.
\end{thm}

\begin{proof} 
First, we assume that $\Omega=\mathbb{D}$.
Let $T_{k,k}= M^{(\lambda_k)*}_z, 1\leq k\leq n $ acting on $\mathbb{A}_{\lambda_k}^2$ in Lemma \ref{homo1},  and $0<\lambda_{k+1}-\lambda_k<2$, $1\leq k\leq n-1$. 
It is well known that $T_{k,k}$ is a backward shift operator with weighted sequence $\displaystyle \left\{\left(\frac{n}{n+\lambda_k-1}\right)^{\frac{1}{2}}\right\}^{\infty}_{n=1}$.  
We define
$T_{k,k+1}, 1\leq k\leq n-1$ to be
$$T_{k,k+1}(e^{(k+1)}_n)= \frac{\prod\limits_{j=1}^n \displaystyle \left(\frac{j}{j+\lambda_{k+1}-1}\right)^{\frac{1}{2}}}{\prod\limits_{j=1}^n\displaystyle \left(\frac{j}{j+\lambda_{k}-1}\right)^{\frac{1}{2}}}e^{(k)}_n,$$
where $\{e^{(k)}_n\}^{\infty}_{n=1}$ be an orthogonal normal basis of the Hilbert space $\mathbb{A}_{\lambda_k}^2.$ 
Let 
$$T_{k,j}=T_{k,k+1}T_{k+1,k+2}\cdots T_{j-1,j}, j>k,$$
and $T_{k,j}=0, j<k$. 
Then by a directly computation,  we can see that 
$$T_{k,k}T_{k,k+1}=T_{k,k+1}T_{k+1,k+1}, \quad T_{k,k}T_{k,j}=T_{k,j}T_{j+1,j+1}, k<j.$$

By Lemma \ref{homo2}, we can see that $(T_{k,k},T_{j,j})$ satisfy the Property (H), for any $k<j$. 
Thus, we can see that $T\in  \mathcal{CFB}_n(\mathbb{D})$.
Notice that each $T_{k,k+1}$ is non-zero. 
By Lemma \ref{SI}, $T$ is a strongly irreducible operator. 

Notice that $T_{k,k}T^*_{k,k}=I_{\mathcal{H}_k}-K_1^{(k)}, $ and $T^*_{k,k}T_{k,k}=I_{\mathcal{H}_k}-K_2^{(k)}$, where $K_1^{(k)}$ and $K_2^{(k)}$ are both rank-one operators, we can see that each $T_{k,k}$ is essentially normal operator. 
Since 
$$\lim\limits_{n\rightarrow \infty} \frac{\prod\limits_{j=1}^n\left(\frac{j}{j+\lambda_{k+1}-1}\right)^{\frac{1}{2}}}{\prod\limits_{j=1}^n\left(\frac{j}{j+\lambda_{k}-1}\right)^{\frac{1}{2}}}=0, \quad \text{when}\quad \lambda_k<\lambda_{k+1},$$ 
then each $T_{k,k+1}$ is compact operator. 
Thus each $T_{k,j}$ is also compact. 
We can see that $T\in \mathcal{NCFB}_n(\mathbb{D})$. 

For simply connected region $\Omega\subset \mathbb{C}$, by Riemann mapping theorem, there exists an analytic function $f:\Omega  \rightarrow \mathbb{D}$ such that $f$ is one-to-one and $f(z_0)=0, f'(z_0)>0$ for some $z_0\in \Omega$. 
By a similar argument, we can see that when one replace $f(T)$ instead of $T$ in the previous proceed, the 
same conclusion is  also valid. 
In fact, let  $T\in \mathcal{NCFB}_n(\mathbb{D})$ defined above. 
Then we have that 
$$f(T)=\left ( \begin{smallmatrix}f(T_{1,1}) & f^{\prime}(T_{1,1})T_{1,2}& f_{1,3}(f(T_{1,1}))T_{1,2}T_{2,3}& \cdots & f_{1,n}(f(T_{1,1}))T_{1,2}\cdots T_{n-1,n}\\
0&f(T_{2,2})&f^{\prime}(T_{2,2})T_{2,3}&\cdots&f_{2,n}(f(T_{2,2}))T_{2,3}\cdots T_{n-1,n}\\
\vdots&\ddots&\ddots&\ddots&\vdots\\
0&\cdots&0&f(T_{n-1,n-1})&f^{\prime}(T_{n-1,n-1})T_{n-1,n}\\
0&\cdots&\cdots&0&f(T_{n,n})
\end{smallmatrix}\right ),$$
where $f_{k,j}\in \mathcal{H}^{\infty}(\Omega).$  
Since $T$ is essentially normal, that means $\m{Q}(T)\m{Q}(T^*)=\m{Q}(T^*)\m{Q}(T)$. 
It follows that 
$$\m{Q}(f(T))\m{Q}(f^*(T))=\m{Q}(f^*(T))\m{Q}(f(T)).$$
Thus, $f(T)$ is also essentially normal.  
Since each $T_{k,j}, k<j$ is  a compact operator, then we have that  $f_{k,j}(f(T_{k,k})T_{k,k+1}\cdots T_{j-1,j}$ is also a compact operator.  
Since $T\in \mathcal{NCFB}_n(\mathbb{D})$, then it is easy to see that the condition (2) in Definition \ref{CFB1} is also satisfied. 
Thus, we only need to prove that $f(T)$ satisfies the Property $(H)$ when  $0<\lambda_{k+1}-\lambda_k<2$. 

Without loss of generality, we only prove that if $f(T_{1,1})Y=Yf(T_{2,2})$ and $Y=f(T_{1,1})X-Xf(T_{2,2})$, then $Y=0$.
Suppose that $f(z)=\sum\limits_{k=0}^{\infty}a_kz^k$ and $a_0=0$. 
In fact, if $a_0\neq 0$, then we can consider $f(T_{1,1})-a_0I$ and $f(T_{2,2})-a_0I$. 
Thus, we can assume that 
\begin{equation}\label{uper1}f(T_{1,1})=\begin{pmatrix}
0& a_{1,2} & a_{1,3}&\cdots&a_{1,k}&\cdots\\
&0&a_{2,3}&\cdots&a_{2,k} &\cdots\\
&&\ddots&\ddots&\vdots&\cdots\\
&&&0&a_{k-1,k}&\cdots\\
&&&&0&\ddots\\
&&&&&\ddots\\
\end{pmatrix}, f(T_{2,2})=\begin{pmatrix}
0& b_{1,2} & b_{1,3}&\cdots&b_{1,k}&\cdots\\
&0&b_{2,3}&\cdots&b_{2,k} &\cdots\\
&&\ddots&\ddots&\vdots&\cdots\\
&&&0&b_{k-1,k}&\cdots\\
&&&&0&\ddots\\
&&&&&\ddots\\
\end{pmatrix}, 
\end{equation}
where $a_{k,k+1}=a_1\sqrt{\frac{k}{k+\lambda_1-1}}$ and $b_{k,k+1}=a_1\sqrt{\frac{k}{k+\lambda_2-1}}$. By Lemma \ref{unhappy}, we have that 
$$x_{k+1,k}=\frac{n \prod\limits^{k}_{i=1}b_{i,i+1}}{\prod\limits^{k+1}_{i=1}a_{i,i+1}}=\frac{n \prod\limits^{k}_{i=1}a_1\sqrt{\frac{i}{i+\lambda_2-1}}}{\prod\limits^{k+1}_{i=1}a_1\sqrt{\frac{i}{i+\lambda_1-1}}}.$$
By Lemma \ref{homo2} and Lemma 4.3 in \cite{JJM17}, we know that $\lim\limits_{k\rightarrow \infty} x_{k,k+1}=\infty$. Thus, $Y=0$.
\end{proof}

\subsection{($\m{U}+\m{K}$)-Orbits}
In \cite{JJK19}, the first two authors and Keshari prove that $\m{CFB}_n(\Omega)$ is norm dense in $\m{B}_n(\Omega)$.
Since it is quite difficult to describe ($\m{U+K}$)-orbit of an arbitrary operator, we would be interested in obtain a similar result for essentially normal operators. 
But by the results in \cite{JiaWan98} and Theorem \ref{example}, we see that the set $\m{NCFB}_n(\Omega)$ contains many nontrivial operators.

\begin{defn}
Suppose $\Omega$ is simply connected.
Let $\mathcal{A}_n(\Omega)$ denote the class of operators $S$, each of which satisfies that 
 \begin{enumerate}
\item[(1)] $S$ is essentially normal;
\item[(2)] $\sigma(S)=\overline{\Omega}$, $\rho_F(S)\cap \sigma(S)=\Omega,$
\item[(3)] $ind(S-w)=n, w\in \Omega$, $\dim\Ker(S-w)=0, w\in \Omega$.
\end{enumerate}

\end{defn}

\begin{rem}
By the definition of the Cowen-Douglas operator, it is clear that $\m{EN}\cap \m{B}_n(\Omega)\subset \m{A}_n(\Omega)$.
\end{rem}

\begin{prop}\cite{JiaWan98}\label{JWL}
Suppose $\Omega$ is simply connected. Let $S\in \mathcal{A}_n(\Omega)$. Then the closure of ($\m{U+K}$)-orbits of $S$ consists of all operators $A$ satisfying 
 \begin{enumerate}
\item[(1)] $A$ is essentially normal;
\item[(2)] $\sigma(A)=\overline{\Omega}$, $\rho_F(A)\cap \sigma(A)=\Omega,$
\item[(3)] $ind(w-A)=n, w\in \Omega$,
\end{enumerate}
 where $\rho_{F}(A)$ denotes the Fredholm domain of $A$.
\end{prop}

\begin{thm}\label{thm:ukorbit}
Suppose that $\Omega$ is simply connected and $T$ is essentially normal.
If 
\begin{itemize}
\item[(1)] $\sigma(T)=\overline{\Omega}$, $\rho_F(T)\cap \sigma(A)=\Omega,$
\item[(2)] $ind(w-T)=n, w\in \Omega$,
\end{itemize}
then $T\in \overline{\m{NCFB}_n(\Omega)}$.
\end{thm}
\begin{proof}
It is directly from Theorem \ref{example} and Proposition \ref{JWL}.
\end{proof}

\section{Similarity Invariants}

In this section, we will study the similarity invariants of the class $\m{NCFB}_n(\Omega)$ of essentially normal Cowen-Douglas operators.
We begin with some technique lemmas.

\begin{lem}\label{J21}
Let $T=(T_{j,k})_{j,k=1}^n$, $\widetilde{T}=(\widetilde{T}_{j,k})_{j,k=1}^n\in \m{NCFB}_n(\Omega).$ 
If $T_{j,j}=\widetilde{T}_{j,j},$ $T_{j,j+1}=\widetilde{T}_{j,j+1}$, then there exists $K\in \m{K}(\m{H})$ such that $X=I+K$ is invertible and $XT=\widetilde{T}X.$ 
\end{lem}

\begin{proof} 
To find $K$, we need to solve the equation 
\begin{equation}\label{mle1}
(I+K)T=\widetilde{T}(I+K).
\end{equation} 
Set $X:=I+K$, where 
$$K=\left ( \begin{smallmatrix}0 & K_{1,2}& K_{1,3}& \cdots & K_{1,n}\\
0&0&K_{2,3}&\cdots&K_{2,n}\\
\vdots&\ddots&\ddots&\ddots&\vdots\\
0&\cdots&0&0&K_{n-1,n}\\
0&\cdots&\cdots&0&0
\end{smallmatrix}\right ).$$
From Equation (\ref{mle1}), we have
\begin{eqnarray}\label{4.1}
&&\left ( \begin{smallmatrix}1 & K_{1,2}& K_{1,3}& \cdots & K_{1,n}\\
0&1&K_{2,3}&\cdots&K_{2,n}\\
\vdots&\ddots&\ddots&\ddots&\vdots\\
0&\cdots&0&1&K_{n-1,n}\\
0&\cdots&\cdots&0&1
\end{smallmatrix}\right )
\left ( \begin{smallmatrix}T_{1,1} & T_{1,2}& {T}_{1,3}& \cdots & {T}_{1,n}\\
0&T_{2,2}&T_{2,3}&\cdots& {T}_{2,n}\\
\vdots&\ddots&\ddots&\ddots&\vdots\\
0&\cdots&0&T_{n-1,n-1}&T_{n-1,n}\\
0&\cdots&\cdots&0&T_{n,n}
\end{smallmatrix}\right ) \nonumber\\
&=&\left ( \begin{smallmatrix}T_{1,1} & T_{1,2}& \widetilde{T}_{1,3}& \cdots & \widetilde{T}_{1,n}\\
0&T_{2,2}&T_{2,3}&\cdots&\widetilde{T}_{2,n}\\
\vdots&\ddots&\ddots&\ddots&\vdots\\
0&\cdots&0&T_{n-1,n-1}&T_{n-1,n}\\
0&\cdots&\cdots&0&T_{n,n}
\end{smallmatrix}\right )\left ( \begin{smallmatrix}1 & K_{1,2}& K_{1,3}& \cdots & K_{1,n}\\
0&1&K_{2,3}&\cdots&K_{2,n}\\
\vdots&\ddots&\ddots&\ddots&\vdots\\
0&\cdots&0&1&K_{n-1,n}\\
0&\cdots&\cdots&0&1
\end{smallmatrix}\right ).
\end{eqnarray}
To find $K_{j,k}$ for $1\leq j<k\leq n$, we need the following steps.

{\bf Step I.}
For $1\leq j\leq n-1$, by equating the $(j,j+1)$ entry of Equation (\ref{4.1}), we have 
$$T_{j,j+1}+ K_{j,j+1}T_{j+1,j+1}= T_{j,j} K_{j,j+1}+ T_{j,j+1},$$ 
i.e. $ K_{j,j+1}T_{j+1,j+1}= T_{j,j} K_{j,j+1}$.

For $1\leq j \leq n-2$, by comparing $(j,j+2)$ entry of Equation (\ref{4.1}), we have
\begin{eqnarray}\label{l4e1}
T_{j,j+2}+K_{j,j+1}T_{j+1,j+2}+K_{j,j+2}T_{j+2,j+2}=T_{j,j}K_{j,j+2}+T_{j,j+1}K_{j+1,j+2}+\widetilde{T}_{j,j+2}.\end{eqnarray}
We assume that 
\begin{align}\label{assumption1}
K_{j,j+2}T_{j+2,j+2}=T_{j,j}K_{j,j+2}, \quad 1\leq j\leq n-2.
\end{align}
Then
\begin{align}\label{l4e31}
T_{j,j+2}+K_{j,j+1}T_{j+1,j+2}=T_{j,j+1}K_{j+1,j+2}+\widetilde{T}_{j,j+2}.
\end{align}
By the definition of $T$ and $\widetilde{T}$, we have
\begin{align}\label{l4e41}
C_{j,j+2}T_{j,j+1}T_{j+1, j+2}+K_{j,j+1}T_{j+1,j+2}=T_{j,j+1}K_{j+1,j+2}+\tilde{C}_{j,j+2}T_{j,j+1}T_{j+1, j+2}.
\end{align}
Let $K_{j,j+1}=C^{(1)}_{j,j+1}T_{j, j+1}$, where $C^{(1)}_{j,j+1}\in \{T_{j,j}\}'$.
Then we may assume that $K_{j,j+1}=\tilde{C}_{j,j+2}T_{j,j+1}-C_{j,j+2}T_{j,j+1}+ T_{j,j+1}C^{(1)}_{j+1,j+2}$.
We also assume that $K_{n-1,n}=T_{n-1,n}$, i.e. $C^{(1)}_{n-1,n}=1$ and we will solve for $K_{j,j+2}$ in the next step.

{\bf Step II.}
By comparing $(j,j+3)$-entry of Equation (\ref{4.1}), we get
 \begin{equation}\label{l4e3}
 \begin{aligned}
&T_{j,j+3}+ K_{j,j+1}T_{j+1,j+3}+K_{j,j+2}T_{j+2,j+3}+K_{j,j+3}T_{j+3,j+3}\\
=&T_{j,j}K_{j,j+3}+T_{j,j+1}K_{j+1,j+3}+\widetilde{T}_{j,j+2}K_{j+2,j+3}+\widetilde{T}_{j,j+3}.
\end{aligned}
 \end{equation}
We assume that 
\begin{align}\label{assumption2}
T_{j,j}K_{j,j+3}=K_{j,j+3}T_{j+3,j+3}, 1\leq j\leq n-3,
\end{align} 
then from Equation (\ref{l4e3}) we have  
 \begin{equation}\label{l4e4}
 \begin{aligned}
&T_{j,j+3}+ K_{j,j+1}T_{j+1,j+3}+K_{j,j+2}T_{i+2,i+3}\\
=&T_{i,i+1}K_{i+1,i+3}+\widetilde{T}_{i,i+2}K_{i+2,i+3}+\widetilde{T}_{j,j+3}.
\end{aligned}
 \end{equation}
Assume that $K_{j,j+2}=C^{(2)}_{j,j+2}T_{j,j+1}T_{j+1, j+2}$, where $C^{(2)}_{j,j+2}\in \{T_{j,j}\}'$.
By using the argument in {\bf Step I.}, we can solve for $K_{j,j+2}$ and assume that $K_{n-2,n}=T_{n-2, n-1}T_{n-1, n}$.

Iterating the arguments in {\bf Step II.}, we can solve for $K_{j,k}$, $1\leq j<k\leq n$ and it is clear that $K_{j,k}$ are compact.
\end{proof}

\begin{lem}\cite{JiaWan98}\label{weight1}
Let  $T\in \m{B}_1(\Omega)$. 
For any $a_0\in \Omega$, there exists an orthogonal normal basis $\{e_n\}_{n=0}^{\infty}$ of ${\mathcal H}$  and $r>0$ such that $T$ admits the upper-triangular matrix representation 
\begin{equation}\label{uper2}T=\begin{pmatrix}
a_0& a_{1,2} & a_{1,3}&\cdots&a_{1,k}&\cdots\\
&a_0&a_{2,3}&\cdots&a_{2,k} &\cdots\\
&&\ddots&\ddots&\vdots&\cdots\\
&&&a_0&a_{k-1,k}&\cdots\\
&&&&a_0&\ddots\\
&&&&&\ddots\\
\end{pmatrix}
\end{equation}
with respect to $\{e_n\}_{n=0}^{\infty}$ and $\left|a_{k-1,k}\right|>r, $ for all $k=2,3,\cdots.$ 
\end{lem}

\begin{lem} \label{unhappy}
Let $T_1$ and $T_2\in \m{B}_1(\Omega)$.  
Suppose $T_1$ and $T_2$ admit the following upper-triangular matrix representation mentioned in Lemma \ref{weight1}, 
\begin{equation}\label{uper3}T_1=\begin{pmatrix}
0& a_{1,2} & a_{1,3}&\cdots&a_{1,k}&\cdots\\
&0&a_{2,3}&\cdots&a_{2,k} &\cdots\\
&&\ddots&\ddots&\vdots&\cdots\\
&&&0&a_{k-1,k}&\cdots\\
&&&&0&\ddots\\
&&&&&\ddots\\
\end{pmatrix}, T_2=\begin{pmatrix}
0& b_{1,2} & b_{1,3}&\cdots&b_{1,k}&\cdots\\
&0&b_{2,3}&\cdots&b_{2,k} &\cdots\\
&&\ddots&\ddots&\vdots&\cdots\\
&&&0&b_{k-1,k}&\cdots\\
&&&&0&\ddots\\
&&&&&\ddots\\
\end{pmatrix}. 
\end{equation}
Suppose $X,Y\in\m{L}(\m{H})$ and  admit the matrix representation $X=(x_{k,j})_{k,j=0}^\infty, Y=(y_{k,j})_{k,j=0}^\infty, $ $y_{k,j}=0, k>j.$  
If $T_1Y=YT_2$ and $Y=T_1X-XT_2$ both hold, then 
\begin{equation}
x_{k+1,k}=\frac{n \prod\limits^{k}_{\ell=1}b_{\ell,\ell+1}}{\prod\limits^{k+1}_{\ell=1}a_{\ell,\ell+1}}.
\end{equation}
\end{lem}

\begin{proof} 
Since $T_1Y=YT_2$ and $Y=(y_{k,j}), $ $y_{k,j}=0, k>j.$ 
Then we have that $y_{k,k}=\frac{\prod\limits^k_{\ell=1}b_{\ell,\ell+1}}{\prod\limits^k_{\ell=1}a_{\ell,\ell+1}}y_{1,1}$. 
Since $T_1X-XT_2=Y$, then we have that $x_{k,j}=0, k-j>1,$ and 
$$a_{1,2}x_{2,1}=y_{1,1},\quad  a_{2,3}x_{3,2}-x_{2,1}b_{1,2}=y_{2,2},\quad  \cdots, \quad a_{k,k+1}x_{k+1,k}-x_{k,k-1}b_{k-1,k}=y_{k,k}, \cdots.$$
It follows that $x_{k+1,k}=\frac{n \prod\limits^{k}_{l=1}b_{l,l+1}}{\prod\limits^{k+1}_{l=1}a_{l,l+1}}.$ 
\end{proof}

\begin{lem}\label{order}
Let $T\in\mathcal{NCFB}_n(\Omega)$. 
Then we have that $\Ker\tau_{T_{j,j}, T_{\ell,\ell}}=\{0\}, j>\ell.$
\end{lem}

\begin{proof} 
By Lemma \ref{weight1}, without loss of generality, we assume that $0\in\Omega$ and there exists $\{e_k\}^{\infty}_{k=1}$ be an orthogonal normal basis of $\mathcal{H}$ such that $T_{\ell,\ell}$ and $T_{j,j}$ have the following the matrix representations according to $\{e_k\}^{\infty}_{k=1}.$
\begin{equation}\label{uper4}
T_{\ell,\ell}=\left(\begin{smallmatrix}
0& a^\ell_{1,2} & a^\ell_{1,3}&\cdots&a^\ell_{1,k}&\cdots\\
&0&a^\ell_{2,3}&\cdots&a^\ell_{2,k} &\cdots\\
&&\ddots&\ddots&\vdots&\cdots\\
&&&0&a^\ell_{k-1,k}&\cdots\\
&&&&0&\ddots\\
&&&&&\ddots\\
\end{smallmatrix}\right ), T_{j,j}=\left(\begin{smallmatrix}
0& a^{j}_{1,2} & a^{j}_{1,3}&\cdots&a^{j}_{1,k}&\cdots\\
&0&a^{j}_{2,3}&\cdots&a^{j}_{2,k} &\cdots\\
&&\ddots&\ddots&\vdots&\cdots\\
&&&0&a^{j}_{k-1,k}&\cdots\\
&&&&0&\ddots\\
&&&&&\ddots\\
\end{smallmatrix}\right ).
\end{equation}
Furthermore, $|a^\ell_{k,k+1}| >r$ and $|a^j_{k,k+1}|>r $ for some $r>0$ and any $k>2$. 

Since $T_{\ell,\ell}T_{\ell,j}=T_{\ell,j}T_{j,j}, 1\leq \ell\leq j-1$, by comparing the elements in the matrix on both sides of the equation, we have that 
$$T_{i,j}=\left ( \begin{matrix}
t_{1,1} & t_{1,2} &\cdots&t_{1,k}&\cdots\\
&t_{2,2}&\cdots&t_{2,k} &\cdots\\
&&\ddots&\vdots&\cdots\\
&&&t_{k,k}&\ddots\\
&&&&\ddots\\
\end{matrix} \right ),$$
where $t_{k,k}=\frac{\prod\limits^{k}_{l=1}a^{j}_{l,l+1}}{\prod\limits^k_{l=1}a^\ell_{l,l+1}}t_{1,1}$. 
Without loss of generality, assume that $t_{1,1}\neq 0$. 
By Theorem \ref{newl1},  each operator $T_{i,j}$ is a compact operator.  
Thus, we can assume that 
$$\lim\limits_{k\rightarrow \infty} t_{k,k}=\lim\limits_{k\rightarrow \infty}\frac{\prod\limits^{k}_{l=1}a^{j}_{l,l+1}}{\prod\limits^k_{l=1}a^\ell_{l,l+1}}t_{1,1}=0.$$
Now assume there exists a bounded operator $Y$ such that $T_{j,j}Y=YT_{\ell,\ell}$, then by a similar computation, we have that 
$$Y=\left ( \begin{matrix}
y_{1,1} & y_{1,2} &\cdots&y_{1,k}&\cdots\\
&y_{2,2}&\cdots&y_{2,k} &\cdots\\
&&\ddots&\vdots&\cdots\\
&&&y_{k,k}&\ddots\\
&&&&\ddots\\
\end{matrix} \right ),$$ 
where $y_{k,k}=\frac{\prod\limits^{k}_{l=1}a^{\ell}_{l,l+1}}{\prod\limits^k_{l=1}a^{j}_{l,l+1}}y_{1,1}.$ Thus, we have that $\lim\limits_{k\rightarrow \infty}y_{k,k}=\infty.$ 
This means $y_{1,1}=0$ and also $y_{k,k}=0, 1\leq k$. 
Then we have that 
\begin{align*}
&\left(\begin{smallmatrix}
0& a^{j}_{1,2} & a^{j}_{1,3}&\cdots&a^{j}_{1,k}&\cdots\\
&0&a^{j}_{2,3}&\cdots&a^{j}_{2,k} &\cdots\\
&&\ddots&\ddots&\vdots&\cdots\\
&&&0&a^{j}_{k-1,k}&\cdots\\
&&&&0&\ddots\\
&&&&&\ddots\\
\end{smallmatrix}\right )\left ( \begin{matrix}
0 & y_{1,2} &\cdots&y_{1,k}&\cdots\\
&0&\cdots&y_{2,k} &\cdots\\
&&\ddots&\vdots&\cdots\\
&&&0&\ddots\\
&&&&\ddots\\
\end{matrix} \right )\\
=&\left ( \begin{matrix}
0 & y_{1,2}&\cdots&y_{1,k}&\cdots\\
&0&\cdots&y_{2,k} &\cdots\\
&&\ddots&\vdots&\cdots\\
&&&0&\ddots\\
&&&&\ddots\\
\end{matrix} \right )\left(\begin{smallmatrix}
0& a^\ell_{1,2} & a^\ell_{1,3}&\cdots&a^\ell_{1,k}&\cdots\\
&0&a^\ell_{2,3}&\cdots&a^\ell_{2,k} &\cdots\\
&&\ddots&\ddots&\vdots&\cdots\\
&&&0&a^\ell_{k-1,k}&\cdots\\
&&&&0&\ddots\\
&&&&&\ddots\\
\end{smallmatrix}\right ).
\end{align*}
By comparing the coefficients of matrices in both sides, we have that 
$$y_{2,3}=y_{1,2}\frac{a^\ell_{2,3}}{a^j_{1,2}},\quad  y_{3,4}=y_{2,3}\frac{a^\ell_{3,4}}{a^j_{2,3}}\quad , \cdots, \quad  y_{k,k+1}=y_{k-1,k}\frac{a^\ell_{k,k+1}}{a^j_{k-1,k}}.$$
Thus, $y_{k,k+1}=\displaystyle \frac{\prod\limits_{l=1}^{k-1}a^\ell_{l,l+1}}{\prod\limits_{l=1}^{k-1}a^j_{l,l+1}}\left(\frac{a^\ell_{k,k+1}}{a^\ell_{1,2}}\right)y_{1,2}$.  
Since $\lim\limits_{k\rightarrow \infty}\frac{\prod\limits_{l=1}^{k-1}a^\ell_{l,l+1}}{\prod\limits_{l=1}^{k-1}a^j_{l,l+1}}=0$ and $|a^\ell_{k,k+1}|>r>0$. 
It follows that $y_{1,2}=0$. 
Thus, we have that $y_{k,k+1}=0,$ for any $k>0$. 
 
 By an inductive proof, we can assume that $y_{k,k+k^{\prime}}=0, k^{\prime}<s, k=1,2\cdots. $ 
 By a similar argument, we have that 
 $$y_{k,k+s}=\frac{\prod\limits_{l=1}^{k-1}a^\ell_{l,l+1}}{\prod\limits_{l=1}^{k-1}a^j_{l,l+1}}\left(\frac{\prod\limits_{m=1}^sa^i_{k,k+m}}{\prod\limits_{m=1}^s a^\ell_{m,m+1}}\right)y_{1,1+s}.$$
 Then we also have that $y_{1,1+s}=0$ and $y_{k,k+s}=0, k\geq 1.$ 
 That means $Y=0$. 
Thus, $\Ker\tau_{T_{j,j}, T_{\ell,\ell}}=\{0\}.$
\end{proof}

\begin{prop}\label{pjjkm}
Let  $T,\widetilde{T}\in\mathcal{NCFB}_n(\Omega)$.
Suppose that $X\in\m{L}(\m{H})$ is invertible such that $XT=\widetilde{T}X$.
Then $X$ and $X^{-1}$ are upper triangular.
 \end{prop}

\begin{proof}
Let $X$ be the invertible operator which intertwines $T$ and $\tilde{T}$. Set $Y=X^{-1}$ and $X=(X_{i,j})_{i,j=1}^n, Y=(Y_{i,j})_{i,j=1}^n$. 
Suppose that 
\begin{equation}\label{2.1}
\left ( \begin{smallmatrix}
X_{1,1} & X_{1,2} &\cdots&X_{1,n}\\
X_{2,1}&X_{2,2}&\cdots&X_{2,n} \\
\vdots&\vdots&\ddots&\vdots\\
X_{n,1}&X_{n,2}&\cdots&X_{n,n}\\
\end{smallmatrix} \right ) \left ( \begin{smallmatrix}
T_{1,1} & T_{1,2} &\cdots&T_{1,n}\\
&T_{2,2}&\cdots&T_{2,n} \\
&&\ddots&\vdots\\
&&&T_{n,n}\\
\end{smallmatrix} \right )=\left ( \begin{smallmatrix}
\tilde{T}_{1,1} & \tilde{T}_{1,2} &\cdots&\tilde{T}_{1,n}\\
&\tilde{T}_{2,2}&\cdots&\tilde{T}_{2,n} \\
&&\ddots&\vdots\\
&&&\tilde{T}_{n,n}\\
\end{smallmatrix} \right )\left ( \begin{smallmatrix}
X_{1,1} & X_{1,2} &\cdots&X_{1,n}\\
X_{2,1}&X_{2,2}&\cdots&X_{2,n} \\
\vdots&\vdots&\ddots&\vdots\\
X_{n,1}&X_{n,2}&\cdots&X_{n,n}\\
\end{smallmatrix} \right ), \end{equation}
\begin{equation}\label{2.2}
 \left ( \begin{smallmatrix}
T_{1,1} & T_{1,2} &\cdots&T_{1,n}\\
&T_{2,2}&\cdots&T_{2,n} \\
&&\ddots&\vdots\\
&&&T_{n,n}\\
\end{smallmatrix} \right )\left ( \begin{smallmatrix}
Y_{1,1} & Y_{1,2} &\cdots&Y_{1,n}\\
Y_{2,1}&Y_{2,2}&\cdots&Y_{2,n} \\
\vdots&\vdots&\ddots&\vdots\\
Y_{n,1}&Y_{n,2}&\cdots&Y_{n,n}\\
\end{smallmatrix} \right )=\left ( \begin{smallmatrix}
Y_{1,1} & Y_{1,2} &\cdots&Y_{1,n}\\
Y_{2,1}&Y_{2,2}&\cdots&Y_{2,n} \\
\vdots&\vdots&\ddots&\vdots\\
Y_{n,1}&Y_{n,2}&\cdots&Y_{n,n}\\
\end{smallmatrix} \right )\left ( \begin{smallmatrix}
\tilde{T}_{1,1} & \tilde{T}_{1,2} &\cdots&\tilde{T}_{1,n}\\
&\tilde{T}_{2,2}&\cdots&\tilde{T}_{2,n} \\
&&\ddots&\vdots\\
&&&\tilde{T}_{n,n}\\
\end{smallmatrix} \right ).\end{equation}
By Equations (\ref{2.1}) and (\ref{2.2}), we have that 
$$ X_{n,1}T_{1,1}=\tilde{T}_{n,n}X_{n,1},\quad T_{n,n}Y_{n,1}=Y_{n,1}\tilde{T}_{1,1}.$$ 
Then we have that $X_{n,1}T_{1,n}Y_{n,1}\in \Ker \tau_{\tilde{T}_{n,n}, \tilde{T}_{1,1}}=\{0\}$.  Since $T_{1,1}T_{1,n}=T_{1,n}T_{n,n}$, then we know that 
$T_{1,n}$ have dense range. 
Furthermore, since $X_{n,1}$ and $Y_{n,1}$ intertwine two operators in $\m{B}_1(\Omega)$, we know $X_{n,1}$ and $Y_{n,1}$ are zero operators or the operators with dense range. 
Thus, we can see that $X_{n,1}=0$ or $Y_{n,1}=0$. 
Now assume that $X_{n,1}=0, Y_{n,1}\neq 0$, then we have that $\Ker \tau_{T_{n,n}, \tilde{T}_{1,1}}\neq \{0\}$. 
Since $\Ker \tau_{T_{2,2}, T_{n,n}}\neq \{0\}$, $\Ker\tau_{T_{n,n}, T_{2,2}}=\{0\}$ and $\Ker \tau_{\tilde{T}_{1,1}, \tilde{T}_{n,n}}\neq \{0\}$, $\Ker \tau_{\tilde{T}_{n,n}, \tilde{T}_{1,1}}= \{0\}$, by a similar argument, we have that 
\begin{equation}\label{2.3}  
\Ker \tau_{ \tilde{T}_{n,n},T_{2,2}}=\{0\}.
\end{equation}
Since $X_{n,1}=0$, then we have that 
\begin{equation}\label{2.4}
\left ( \begin{smallmatrix}
X_{1,1} & X_{1,2} &\cdots&X_{1,n}\\
X_{2,1}&X_{2,2}&\cdots&X_{2,n} \\
\vdots&\vdots&\ddots&\vdots\\
0&X_{n,2}&\cdots&X_{n,n}\\
\end{smallmatrix} \right ) \left ( \begin{smallmatrix}
T_{1,1} & T_{1,2} &\cdots&T_{1,n}\\
&T_{2,2}&\cdots&T_{2,n} \\
&&\ddots&\vdots\\
&&&T_{n,n}\\
\end{smallmatrix} \right )=\left ( \begin{smallmatrix}
\tilde{T}_{1,1} & \tilde{T}_{1,2} &\cdots&\tilde{T}_{1,n}\\
&\tilde{T}_{2,2}&\cdots&\tilde{T}_{2,n} \\
&&\ddots&\vdots\\
&&&\tilde{T}_{n,n}\\
\end{smallmatrix} \right )\left ( \begin{smallmatrix}
X_{1,1} & X_{1,2} &\cdots&X_{1,n}\\
X_{2,1}&X_{2,2}&\cdots&X_{2,n} \\
\vdots&\vdots&\ddots&\vdots\\
0&X_{n,2}&\cdots&X_{n,n}\\
\end{smallmatrix} \right ). \end{equation}
Thus, we have that $X_{n,2}T_{2,2}=\tilde{T}_{n,n}X_{n,2}.$ 
By Equation (\ref{2.3}),  $X_{n,2}=0$. 
Going to the next step, we will also have $X_{n,i}=0, 1\leq i\leq n$. However, this is impossible since $X$ is invertible.  
Thus, we will have $Y_{n,1}=0$. 
On the other hand,  we can also prove that the assumption $X_{n,1}\neq 0, Y_{n,1}=0$ does not hold.  By a direction calculation, 
$$X_{n,1}=Y_{n,1}=0.$$
Thus, we have 
$$
\left ( \begin{smallmatrix}
X_{1,1} & X_{1,2} &\cdots&X_{1,n}\\
X_{2,1}&X_{2,2}&\cdots&X_{2,n} \\
\vdots&\vdots&\ddots&\vdots\\
0&X_{n,2}&\cdots&X_{n,n}\\
\end{smallmatrix} \right ) \left ( \begin{smallmatrix}
T_{1,1} & T_{1,2} &\cdots&T_{1,n}\\
&T_{2,2}&\cdots&T_{2,n} \\
&&\ddots&\vdots\\
&&&T_{n,n}\\
\end{smallmatrix} \right )=\left ( \begin{smallmatrix}
\tilde{T}_{1,1} & \tilde{T}_{1,2} &\cdots&\tilde{T}_{1,n}\\
&\tilde{T}_{2,2}&\cdots&\tilde{T}_{2,n} \\
&&\ddots&\vdots\\
&&&\tilde{T}_{n,n}\\
\end{smallmatrix} \right )\left ( \begin{smallmatrix}
X_{1,1} & X_{1,2} &\cdots&X_{1,n}\\
X_{2,1}&X_{2,2}&\cdots&X_{2,n} \\
\vdots&\vdots&\ddots&\vdots\\
0&X_{n,2}&\cdots&X_{n,n}\\
\end{smallmatrix} \right ), 
$$
and 
$$
 \left ( \begin{smallmatrix}
T_{1,1} & T_{1,2} &\cdots&T_{1,n}\\
&T_{2,2}&\cdots&T_{2,n} \\
&&\ddots&\vdots\\
&&&T_{n,n}\\
\end{smallmatrix} \right )\left ( \begin{smallmatrix}
Y_{1,1} & Y_{1,2} &\cdots&Y_{1,n}\\
Y_{2,1}&Y_{2,2}&\cdots&Y_{2,n} \\
\vdots&\vdots&\ddots&\vdots\\
0&Y_{n,2}&\cdots&Y_{n,n}\\
\end{smallmatrix} \right )=\left ( \begin{smallmatrix}
Y_{1,1} & Y_{1,2} &\cdots&Y_{1,n}\\
Y_{2,1}&Y_{2,2}&\cdots&Y_{2,n} \\
\vdots&\vdots&\ddots&\vdots\\
0&Y_{n,2}&\cdots&Y_{n,n}\\
\end{smallmatrix} \right )\left ( \begin{smallmatrix}
\tilde{T}_{1,1} & \tilde{T}_{1,2} &\cdots&\tilde{T}_{1,n}\\
&\tilde{T}_{2,2}&\cdots&\tilde{T}_{2,n} \\
&&\ddots&\vdots\\
&&&\tilde{T}_{n,n}\\
\end{smallmatrix} \right ).$$
Thus, we have that $X_{n,2}T_{2,2}=\tilde{T}_{n,n}X_{n,2}, Y_{n,2}\tilde{T}_{2,2}=T_{n,n}Y_{n,2}.$   By a proof similar to the proof mentioned above, we can also see that the statements 
$X_{n,2}=0, Y_{n,2}\neq 0$ and $X_{n,2}\neq 0, Y_{n,2}=0$ both can not hold. Thus, we also have that $X_{n,2}=Y_{n,2}=0$. 
Keep doing this, we will have that $$X_{n,1}=X_{n,2}=\cdots=X_{n,n-1}=Y_{n,1}=Y_{n,2}=\cdots=Y_{n,n-1}=0.$$ 
Thus, $X_{n,n}T_{n,n}=\tilde{T}_{n,n}X_{n,n}$, $T_{n,n}Y_{n,n}=Y_{n,n}\tilde{T}_{n,n}$. Since $X_{n,n}$ and $Y_{n,n}$ are both non-zero operators, 
we know that $T_{n,n}\sim \tilde{T}_{n,n}.$  Thus, we have that 
\begin{align*}
&\left ( \begin{smallmatrix}
X_{1,1} & X_{1,2} &\cdots&X_{1,n-1}\\
X_{2,1}&X_{2,2}&\cdots&X_{2,n-1} \\
\vdots&\vdots&\ddots&\vdots\\
X_{n-1,1}&X_{n-1,2}&\cdots&X_{n-1,n-1}\\
\end{smallmatrix} \right ) \left ( \begin{smallmatrix}
T_{1,1} & T_{1,2} &\cdots&T_{1,n-1}\\
&T_{2,2}&\cdots&T_{2,n-1} \\
&&\ddots&\vdots\\
&&&T_{n-1,n-1}\\
\end{smallmatrix} \right )\\
=& \left ( \begin{smallmatrix}
\tilde{T}_{1,1} & \tilde{T}_{1,2} &\cdots&\tilde{T}_{1,n-1}\\
&\tilde{T}_{2,2}&\cdots&\tilde{T}_{2,n-1} \\
&&\ddots&\vdots\\
&&&\tilde{T}_{n-1,n-1}\\
\end{smallmatrix} \right )\left ( \begin{smallmatrix}
X_{1,1} & X_{1,2} &\cdots&X_{1,n-1}\\
X_{2,1}&X_{2,2}&\cdots&X_{2,n-1} \\
\vdots&\vdots&\ddots&\vdots\\
X_{n-1,1}&X_{n-1,2}&\cdots&X_{n-1,n-1}\\
\end{smallmatrix} \right ),
\end{align*}
and
\begin{align*}
& \left ( \begin{smallmatrix}
T_{1,1} & T_{1,2} &\cdots&T_{1,n-1}\\
&T_{2,2}&\cdots&T_{2,n-1} \\
&&\ddots&\vdots\\
&&&T_{n-1,n-1}\\
\end{smallmatrix} \right )\left ( \begin{smallmatrix}
Y_{1,1} & Y_{1,2} &\cdots&Y_{1,n-1}\\
Y_{2,1}&Y_{2,2}&\cdots&Y_{2,n-1} \\
\vdots&\vdots&\ddots&\vdots\\
Y_{n-1,1}&Y_{n-1,2}&\cdots&Y_{n-1,n-1}\\
\end{smallmatrix} \right )\\
=&\left ( \begin{smallmatrix}
Y_{1,1} & Y_{1,2} &\cdots&Y_{1,n-1}\\
Y_{2,1}&Y_{2,2}&\cdots&Y_{2,n-1} \\
\vdots&\vdots&\ddots&\vdots\\
Y_{n-1,1}&Y_{n-1,2}&\cdots&Y_{n-1,n-1}\\
\end{smallmatrix} \right )\left ( \begin{smallmatrix}
\tilde{T}_{1,1} & \tilde{T}_{1,2} &\cdots&\tilde{T}_{1,n-1}\\
&\tilde{T}_{2,2}&\cdots&\tilde{T}_{2,n-1} \\
&&\ddots&\vdots\\
&&&\tilde{T}_{n-1,n-1}\\
\end{smallmatrix} \right ).
\end{align*}
Repeating this inductive proof, we can show that $X_{k,j}=Y_{k,j}=0,$ whenever $ k>j$. 
This finishes the proof of this lemma. 
\end{proof}

\begin{lem}\label{Jiangandji}
Let $T, \widetilde{T}\in \mathcal{NCFB}_n(\Omega).$ 
Suppose that $T\sim_{\m{U+K}}\widetilde{T}$ and $T=(T_{j,k})_{j,k=1}^n$, $\widetilde{T}=(\widetilde{T}_{j,k})_{j,k=1}^n$.
Then there exists a diagonal unitary operator $U$ and a compact upper triangular operator $K$ such that $\widetilde{T}(U+K)=(U+K)T.$
\end{lem}

\begin{proof} 
Since $T\sim_{\m{U+K}}\widetilde{T}$, there exists an invertible operator $X\in \m{L}(\m{H})$ such that $XT=\widetilde{T}X$ and $X$ is a sum of a unitary operator and a compact operator. 
Then $\m{Q}(X)$ is unitary in the Calkin algebra $\m{A}(\m{H})$.
By Proposition \ref{pjjkm},  $X$ is upper-triangular operator, i.e.
$$X
=\left ( \begin{smallmatrix}X_{1,1} &X_{1,2}& X_{1,3}& \cdots & X_{1,n}\\
0&X_{2,2}&X_{2,3}&\cdots&X_{2,n}\\
\vdots&\ddots&\ddots&\ddots&\vdots\\
0&\cdots&0&X_{n-1,n-1}&X_{n-1,n}\\
0&\cdots&\cdots&0&X_{n,n}
\end{smallmatrix}\right ).$$
It is then clear that $X_{j,j}$ intertwines $T_{j,j}$ and $\widetilde{T}_{j,j}$ and $X_{j,j}$ are invertible.
Notice that each $\m{Q}(X_{j,j})$ is invertible and $\m{Q}(X)$ is unitary in $\m{A}(\m{H})$, we have that $\m{Q}(X)\m{Q}(X^*)=\m{Q}(X^*)\m{Q}(X)=I$, i.e. 
$$\left ( \begin{smallmatrix}
\m{Q}(X_{1,1}) & \m{Q}(X_{1,2}) &\cdots&\m{Q}(X_{1,n})\\
&\m{Q}(X_{2,2})&\cdots&\m{Q}(X_{2,n}) \\
&&\ddots&\vdots\\
&&&\m{Q}(X_{n,n})\\
\end{smallmatrix} \right )\left ( \begin{smallmatrix}
\m{Q}(X^*_{1,1}) & &&\\
 \m{Q}(X^*_{1,2})&\m{Q}(X^*_{2,2})&& \\
\vdots&\vdots&\ddots& \\
\m{Q}(X^*_{1,n})&\m{Q}(X^*_{2,n})&\cdots&\m{Q}(X^*_{n,n})\\
\end{smallmatrix} \right )=I.$$
It follows that $\m{Q}(X_{1,n})\m{Q}(X^*_{n,n})=0$. 
Since $\m{Q}(X^*_{n,n})$ is invertible, then we have $\m{Q}(X_{1,n})=0$.  
Thus, we will have that 
$\m{Q}(X_{1,n-1})\m{Q}(X_{n-1,n-1})=0,$ and then $\m{Q}(X_{1,n-1})=0$. 
Repeating this arguments, we will have that $\m{Q}(X_{i,j})=0, i<j.$ 
That means $\m{Q}(X)$ is a diagonal matrix and $\m{Q}(X_{i,i})\m{Q}(T_{i,i})=\m{Q}(\widetilde{T}_{j,j})\m{Q}(X_{j,j\widetilde{T}}).$
Let $U_j$ be the polar part of the polar decomposition of $X_{j,j}$.
Then
$$\m{Q}(X_{j,j})=\m{Q}\left(U_j|X_{j,j}|\right)=\m{Q}(U_j)\m{Q}(X_{j,j}^*X_{j,j})^{1/2}=\m{Q}(U_j),$$
Let $U$ be a diagonal matrix whose $(j,j)$-entry is $U_j$ and $K=X-U$.
Then $U$ is unitary and $K$ is compact and the proof is complete.
\end{proof}

\begin{lem}\label{J3}
Let $T,\widetilde{T}\in \mathcal{CFB}_n(\Omega)$. 
Suppose $X\in \m{L}(\m{H})$ is invertible and $X\widetilde{T}=TX$.
We assume that $T=(T_{j,k})_{j,k=1}^n$, $\widetilde{T}=(\widetilde{T}_{j,k})_{j,k=1}^n$, and $X=(X_{j,k})_{j,k=1}^n$ are upper triangular.
Then
\begin{eqnarray*}
&&\left ( \begin{smallmatrix}X_{1,1} & 0& 0& \cdots & 0\\
0&X_{2,2}&0&\cdots&0\\
\vdots&\ddots&\ddots&\ddots&\vdots\\
0&\cdots&0&X_{n-1,n-1}&0\\
0&\cdots&\cdots&0&X_{n,n}
\end{smallmatrix}\right )\left ( \begin{smallmatrix}\widetilde{T}_{1,1} & \widetilde{T}_{1,2}&0& \cdots & 0\\
0&\widetilde{T}_{2,2}&\widetilde{T}_{2,3}&\cdots&0\\
\vdots&\ddots&\ddots&\ddots&\vdots\\
0&\cdots&0&\widetilde{T}_{n-1,n-1}&T_{n-1,n}\\
0&\cdots&\cdots&0&\widetilde{T}_{n,n}
\end{smallmatrix}\right )\\
&=&\left ( \begin{smallmatrix}T_{1,1} & T_{1,2}& 0& \cdots & 0\\
0&T_{2,2}&T_{2,3}&\cdots&0\\
\vdots&\ddots&\ddots&\ddots&\vdots\\
0&\cdots&0&T_{n-1,n-1}&T_{n-1,n}\\
0&\cdots&\cdots&0&T_{n,n}
\end{smallmatrix}\right )\left ( \begin{smallmatrix}X_{1,1} & 0& 0& \cdots & 0\\
0&X_{2,2}&0&\cdots&0\\
\vdots&\ddots&\ddots&\ddots&\vdots\\
0&\cdots&0&X_{n-1,n-1}&0\\
0&\cdots&\cdots&0&X_{n,n}
\end{smallmatrix}\right ).
\end{eqnarray*}
\end{lem}

\begin{proof}
By equating the entries of $X\widetilde{T}=TX$, we get 
$$X_{j,j}\widetilde{T}_{j,j}=T_{j,j}X_{j,j}, \quad 1\leq j\leq n.$$
Let $Y:=X\diag(X^{-1})$.
Then
\begin{eqnarray*}
Y=\left ( \begin{smallmatrix}I &X_{1,2}X^{-1}_{2,2}& X_{1,3}X^{-1}_{3,3}& \cdots & X_{1,n}X^{-1}_{n,n}\\
0&I&X_{2,3}X^{-1}_{3,3}&\cdots&X_{2,n}X^{-1}_{n,n}\\
\vdots&\ddots&\ddots&\ddots&\vdots\\
0&\cdots&0&I&X_{n-1,n}X^{-1}_{n,n}\\
0&\cdots&\cdots&0&I
\end{smallmatrix}\right ).
\end{eqnarray*}

From $T=X\widetilde{T}X^{-1}$, we get 
\begin{align*}
TY=TX\diag(X^{-1})=X\widetilde{T}\diag(X^{-1})= Y\diag(X)\widetilde{T}\diag(X^{-1}),
\end{align*}
which is equivalent to   
\begin{equation}\label{ml2e1}
\begin{aligned}
&\left ( \begin{smallmatrix}T_{1,1} & T_{1,2}& {T}_{1,3}& \cdots & {T}_{1,n}\\
0&T_{2,2}&T_{2,3}&\cdots&{T}_{2,n}\\
\vdots&\ddots&\ddots&\ddots&\vdots\\
0&\cdots&0&T_{n-1,n-1}&T_{n-1,n}\\
0&\cdots&\cdots&0&T_{n,n}
\end{smallmatrix}\right )\left ( \begin{smallmatrix}I &X_{1,2}X^{-1}_{2,2}& X_{1,3}X^{-1}_{3,3}& \cdots & X_{1,n}X^{-1}_{n,n}\\
0&I&X_{2,3}X^{-1}_{3,3}&\cdots&X_{2,n}X^{-1}_{n,n}\\
\vdots&\ddots&\ddots&\ddots&\vdots\\
0&\cdots&0&I&X_{n-1,n}X^{-1}_{n,n}\\
0&\cdots&\cdots&0&I
\end{smallmatrix}\right )\\
=&\left ( \begin{smallmatrix}I &X_{1,2}X^{-1}_{2,2}& X_{1,3}X^{-1}_{3,3}& \cdots & X_{1,n}X^{-1}_{n,n}\\
0&I&X_{2,3}X^{-1}_{3,3}&\cdots&X_{2,n}X^{-1}_{n,n}\\
\vdots&\ddots&\ddots&\ddots&\vdots\\
0&\cdots&0&I&X_{n-1,n}X^{-1}_{n,n}\\
0&\cdots&\cdots&0&I
\end{smallmatrix}\right )
\left ( \begin{smallmatrix}T_{1,1}&X_{1,1}\widetilde{T}_{1,2}X^{-1}_{2,2}&\cdots& \cdots&X_{1,1}\widetilde{T}_{1,n}X^{-1}_{n,n}\\
0&T_{2,2}&X_{2,2}\widetilde{T}_{2,3}X^{-1}_{3,3}&\cdots&X_{2,2}\widetilde{T}_{2,n}X^{-1}_{n,n} \\
\vdots&\ddots&\ddots&\ddots&\vdots\\
0&0&\cdots&0&T_{n,n}
\end{smallmatrix}\right ).
\end{aligned}
\end{equation}

Comparing the $(j,j+1)$-entries of Equation (\ref{ml2e1}), we obtain that 
\begin{eqnarray*}
T_{j,j+1}+T_{j,j}X_{j,j+1}X^{-1}_{j+1,j+1}
=X_{j,j+1}X^{-1}_{j+1,j+1}T_{j+1,j+1}+ X_{j,j}\widetilde{T}_{j,j+1}X^{-1}_{j+1,j+1}.
\end{eqnarray*}
i.e.
\begin{align}\label{4.5}
T_{j,j+1}-X_{j,j}\widetilde{T}_{j,j+1}X^{-1}_{j+1,j+1}
=X_{j,j+1}X^{-1}_{j+1,j+1}T_{j+1,j+1}-T_{j,j}X_{j,j+1}X^{-1}_{j+1,j+1},
\end{align}
denoted by $Y_{j,j+1}$.
Then we consider
\begin{align*}
T_{j,j}Y_{j,j+1}&=T_{j,j}(T_{j,j+1}-X_{j,j}\widetilde{T}_{j,j+1}X^{-1}_{j+1,j+1})\\
=&T_{j,j+1}T_{j+1,j+1}-X_{j,j}\widetilde{T}_{j,j}\widetilde{T}_{j,j+1}X^{-1}_{j+1,j+1}\\
=&T_{j,j+1}T_{j+1,j+1}-X_{j,j}\widetilde{T}_{j,j+1}\widetilde{T}_{j+1,j+1}X^{-1}_{j+1,j+1}\\
=&T_{j,j+1}T_{j+1,j+1}-X_{j,j}\widetilde{T}_{j,j+1}X^{-1}_{j+1,j+1}T_{j+1,j+1}\\
=& (T_{j,j+1}-X_{j,j}\widetilde{T}_{j,j+1}X^{-1}_{j+1,j+1}) T_{j+1,j+1}\\
=& Y_{j,j+1}T_{j+1,j+1},
\end{align*}
i.e. $Y_{j,j+1}\in \Ker \tau_{T_{j,j},T_{j+1,j+1}} \cap \Ran \tau_{T_{j,j},T_{j+1,j+1}} $. 
Recall that $T$ satisfies the Property $(H)$, we see $Y_{j,j+1}=0$, i.e.
$$T_{j,j+1}=X_{j,j}\widetilde{T}_{j,j+1}X^{-1}_{j+1,j+1}, \quad 1\leq j\leq n-1.$$

Finally, we have the following matrix equation:
\begin{eqnarray*}
&&\left ( \begin{smallmatrix}X_{1,1} & 0& 0& \cdots & 0\\
0&X_{2,2}&0&\cdots&0\\
\vdots&\ddots&\ddots&\ddots&\vdots\\
0&\cdots&0&X_{n-1,n-1}&0\\
0&\cdots&\cdots&0&X_{n,n}
\end{smallmatrix}\right )\left ( \begin{smallmatrix}\widetilde{T}_{1,1} & \widetilde{T}_{1,2}&0& \cdots & 0\\
0&\widetilde{T}_{2,2}&\widetilde{T}_{2,3}&\cdots&0\\
\vdots&\ddots&\ddots&\ddots&\vdots\\
0&\cdots&0&\widetilde{T}_{n-1,n-1}&0\\
0&\cdots&\cdots&0&\widetilde{T}_{n,n}
\end{smallmatrix}\right )\\
&=&\left ( \begin{smallmatrix}T_{1,1} & T_{1,2}& 0& \cdots & 0\\
0&T_{2,2}&T_{2,3}&\cdots&0\\
\vdots&\ddots&\ddots&\ddots&\vdots\\
0&\cdots&0&T_{n-1,n-1}&0\\
0&\cdots&\cdots&0&T_{n,n}
\end{smallmatrix}\right )\left ( \begin{smallmatrix}X_{1,1} & 0& 0& \cdots & 0\\
0&X_{2,2}&0&\cdots&0\\
\vdots&\ddots&\ddots&\ddots&\vdots\\
0&\cdots&0&X_{n-1,n-1}&0\\
0&\cdots&\cdots&0&X_{n,n}
\end{smallmatrix}\right ).
\end{eqnarray*}

\end{proof}

\begin{thm}\label{thm:similarity}
Let $T, \widetilde{T}\in \mathcal{NCFB}_n(\Omega)$.
Then  $T\sim_{\m{U+K}}\widetilde{T}$ if and only if  the following statements hold:
\begin{enumerate}
\item[(1)] $C_{\lambda}(E_{diag{T}})(w)=\det[(I+\frac{i\lambda}{2\pi}K_{E_{diag\tilde{T}}}(w)) \oplus  {\mathcal A}(w)]$, for any $w\in \Omega$, $\lambda\in \mathbb{C}$;
\item[(2)] $\frac{\phi_j}{\phi_{j+1}}\theta_{j,j+1}(T)=\theta_{j,j+1}(\widetilde{T})$, for any $1\leq j\leq n-1$,
\end{enumerate} 
where   ${\mathcal A}(w)=\bigoplus\limits_{j=1}^n\frac{\partial^2}{\partial w\partial\overline{w}}\frac{i\lambda}{2\pi}\ln(\phi_j(w)) $,and  $\phi_j(w)$ are functions appeared in Proposition \ref{u+k}.

%  $\Phi(w)$ is a diagonal matrix whose $(j,j)$ entry is $\displaystyle \frac{\partial^2}{\partial w\partial\overline{w}}\ln\left(\phi_j(w)\|\tilde{t}_j(w)\|^2\right)$, $\tilde{t}_j$ is a nonzero section of bundle $E_{\widetilde{T}_{j,j}}$, and $\phi_j(w)=\displaystyle \frac{\left\|X_j(\tilde{t}_j(w))\right\|^2}{\|\tilde{t}_j(w)\|^2}+(1-\alpha_j^2)$ are functions appeared in Proposition \ref{u+k}. 
\end{thm}

\begin{rem}In the main theorem above, when ${\mathcal A}(w)=0, w\in \Omega$, then  the equation in statement (1) turns out to be 
$$C_{\lambda}(E_{diag{T}})(w)=C_{\lambda}(E_{diag\tilde{T}})(w).$$
In this case, $T\sim_u \tilde{T}.$ (See in Theorem  \ref{thm:unitary})
\end{rem}

In the following, we give the proof of Theorem \ref{thm:similarity}.

\begin{proof}
We assume that $0\in \Omega$.
Suppose that  $C_{\lambda}(E_{diag{T}})=\det\left[\left(I+\frac{i\lambda}{2\pi}K_{E_{diag\tilde{T}}}(w)\right) \oplus {\mathcal A}(w)\right]$ for any $w\in \Omega$, $\lambda\in \mathbb{C}$. 
By the definition, we obtain that
$$\prod\limits_{j=1}^{n}\left(1+\lambda \frac{i}{2\pi}K_{T_{j,j}}(w)\right)=\prod\limits_{j=1}^n \left(1-\lambda\frac{i}{2\pi}\frac{\partial^2}{\partial w\partial\overline{w}} \ln(\phi_j(w)\|\tilde{t}_j(w)\|^2)\right),$$
where $\tilde{t}_j$ is a nonzero section of bundle $E_{\widetilde{T}_{j,j}}$, and  $\phi_j(w)=\displaystyle \frac{\left\|X_j(\tilde{t}_j(w))\right\|^2}{\|\tilde{t}_j(w)\|^2}+(1-\alpha_j^2)$ (See in Proposition \ref{u+k}).
By the fundamental theorem of algebra,  there exists a permutation $\Xi$ on $\{1,2,\cdots,n\}$ such that for any $1\leq j\leq n$, we have
$$K_{T_{j,j}}(w)=-\frac{\partial^2}{\partial w\partial\overline{w}} \ln\left(\phi_{\Xi(j)}(w)\|\tilde{t}_{\Xi(j)}(w)\|^2\right).$$
Note that 
\begin{eqnarray*}
K_{T_{j,j}}
&=&-\frac{\partial^2}{\partial w\partial\overline{w}}\ln\left(\phi_{\Xi(j)}(w)\|\tilde{t}_{\Xi(j)}(w)\|^2\right)\\
&=&-\frac{\partial^2}{\partial w\partial\overline{w}}\ln \phi_{\Xi(j)}(w)-\frac{\partial^2}{\partial w\partial\overline{w}}\ln(\|\tilde{t}_{\Xi(j)}(w)\|^2)\\
&=&-\frac{\partial^2}{\partial w\partial\overline{w}}\ln \phi_{\Xi(j)}(w)+K_{\widetilde{T}_{\Xi(j),\Xi(j)}},
\end{eqnarray*}
i.e. $K_{\widetilde{T}_{\Xi(j),\Xi(j)}}-K_{T_{j,j}}=\displaystyle \frac{\partial^2}{\partial w\partial\overline{w}}\ln \phi_{\Xi(j)}(w)$. 
By Lemma \ref{u+k}, we have that $T_j\sim_{\m{U+K}} \widetilde{T}_{\Xi(j),\Xi(j)},$ i.e. there exist invertible operators  $Y_j=U_j+K_j$ such that $U_j$ are unitary, $K_j$ are compact and 
$$(U_j+K_j)\widetilde{T}_{\Xi(j),\Xi(j)}=T_j(U_j+K_j), \quad 1\leq j\leq n. $$

Now we will prove that $\Xi(j)=j, 1\leq j\leq n$. 
We assume that $\Xi(n)<n$. 
Since $(U_n+K_n)\widetilde{T}_{\Xi(n),\Xi(n)}=T_{n,n}(U_n+K_n),$ 
then we have that 
\begin{eqnarray*}
(U_n+K_n)\widetilde{T}_{\Xi(n),\Xi(n)}\widetilde{T}_{\Xi(n),n}&=&T_{n,n}(U_n+K_n)\widetilde{T}_{\Xi(n),n},\\
(U_n+K_n)\widetilde{T}_{\Xi(n),n}\widetilde{T}_{n,n}&=&T_{n,n}(U_n+K_n)\widetilde{T}_{\Xi(n),n}.
\end{eqnarray*}
Notice that 
$$(U_{\Xi^{-1}(n)}+K_{\Xi^{-1}(n)})\widetilde{T}_{n,n}=T_{\Xi^{-1}(n),\Xi^{-1}(n)}(U_{\Xi^{-1}(n)}+K_{\Xi^{-1}(n)}),$$
and $\Xi^{-1}(n)<n$ (Recall $\Xi(n)<n$). Then we have that 
$$
(U_n+K_n)\widetilde{T}_{\Xi(n),n}\widetilde{T}_{n,n}(U_{\Xi^{-1}(n)}+K_{\Xi^{-1}(n)})^{-1}
=T_{n,n}(U_n+K_n)\widetilde{T}_{\Xi(n),n}(U_{\Xi^{-1}(n)}+K_{\Xi^{-1}(n)})^{-1}.
$$
It follows that 
$$
(U_n+K_n)\widetilde{T}_{\Xi(n),n}(U_{\Xi^{-1}(n)}+K_{\Xi^{-1}(n)})^{-1}T_{\Xi^{-1}(n),\Xi^{-1}(n)}=T_{n,n}(U_n+K_n)\widetilde{T}_{\Xi(n),n}(U_{\Xi^{-1}(n)}+K_{\Xi^{-1}(n)})^{-1}.
$$
Thus, 
$$(U_n+K_n)\widetilde{T}_{\Xi(n),n}(U_{\Xi^{-1}(n)}+K_{\Xi^{-1}(n)})^{-1}\in \Ker\tau_{T_{n,n}, T_{\Xi^{-1}(n),\Xi^{-1}(n)}}.$$ Since $\Xi^{-1}(n)<n$, by Lemma \ref{order}, we know that $ \Ker\tau_{T_{n,n}, T_{\Xi^{-1}(n),\Xi^{-1}(n)}}=\{0\}$. 
Thus we have that $\widetilde{T}_{\Xi(n),n}=0$. 
This is a contradiction and $\Xi(n)=n$.
Iterating the process, we can obtain that $\Xi(j)=j$ for $1\leq j\leq n$.
Moreover
$$T_{j,j}=(U_j+K_j)\widetilde{T}_{j,j} (U_j+K_j)^{-1},\quad  j=1,2,\ldots,n.$$

Let $\overline{T}=Y^{-1}TY$, where $Y$ is a diagonal matrix whose $(j,j)$ entry is $Y_j=U_j+K_j$.
The operator $\overline{T}$ has the following form:
$$\overline{T}=\left ( \begin{smallmatrix}\widetilde{T}_{1,1}& Y_1^{-1}T_{1,2}Y_2&\cdots& \cdots& Y_1^{-1}T_{1,n}Y_n\\
0&\widetilde{T}_{2,2}&Y_2^{-1}T_{2,3}Y_3 &\cdots& Y_2^{-1}T_{2,n}Y_n \\
\vdots&\ddots&\ddots&\ddots&\vdots\\
0&0&\cdots&\widetilde{T}_{n-1,n-1}&Y_{n-1}^{-1}T_{n-1,n}Y_n\\
0&0&\cdots&0&\widetilde{T}_{n,n}
\end{smallmatrix}\right ).$$
It is clear that $\overline{T}\sim_{\m{U+K}}T$. 

We define operators $A, B\in\m{L}(\m{H})$ by using matrices:
$$A:=\left( \begin{smallmatrix}\widetilde{T}_{1,1}& Y_1^{-1}T_{1,2}Y_2 &0& \cdots&0\\
0&\widetilde{T}_{2,2}&Y_2^{-1}T_{2,3}Y_3 &0 &0 \\
\vdots&\ddots&\ddots&\ddots&\vdots\\
0&0&\cdots&\widetilde{T}_{n-1,n-1}&Y_{n-1}^{-1}T_{n-1,n}Y_n\\
0&0&\cdots&0&\widetilde{T}_{n,n}
\end{smallmatrix}\right),$$ 
and 
$$B:=\left ( \begin{smallmatrix}\widetilde{T}_{1,1}&\widetilde{T}_{1,2}&0& \cdots&0\\
0&\widetilde{T}_{2,2}&\widetilde{T}_{2,3}&0&0 \\
\vdots&\ddots&\ddots&\ddots&\vdots\\
%0&0&\cdots&T_{n-1,n-1}&(U_{n-1}+K_{n-1,n-1})\widetilde{T}_{n-1,n}(U_n+K_{n,3})^{-1}\\
0&0&\cdots&\widetilde{T}_{n-1,n-1}&\widetilde{T}_{n-1,n}\\
0&0&\cdots&0&\widetilde{T}_{n,n}
\end{smallmatrix}\right ).$$
It is easy to see that 
$$\theta_{j,j+1}(\overline{T})=\theta_{j,j+1}(A),  \quad \theta_{j,j+1}(\widetilde{T})=\theta_{j,j+1}(B), \quad  1\leq j\leq n-1.$$

We claim that $A$ is unitarily equivalent to $B$.  
In fact, by Theorem 3.6 in \cite{JJKM17},  we only need to prove the second fundamental forms of $A$ and $B$ are same. 
Clearly $Y_{j+1}^{-1}(t_{i+1}(\cdot))$ is a non zero section of $E_{\widetilde{T}_{j+1,j+1}}$, 
\begin{eqnarray*}
\theta_{j,j+1}(A)(w)
&=&\frac{\|Y_{j}^{-1}T_{j,j+1}Y_{j+1}Y_{j+1}^{-1}(t_{i+1}(w))\|^2}{\|Y_{j+1}^{-1}(t_{i+1}(w))\|^2}\\
&=&\frac{\|Y_j^{-1}T_{i,i+1}(t_{i+1}(w))\|^2}{\|Y_{j+1}^{-1}(t_{i+1}(w))\|^2}
\end{eqnarray*}
and 
$$\theta_{j,j+1}(B)(w)=\frac{\|\widetilde{T}_{j,j+1}Y_{j+1}^{-1}(t_{j+1}(w))\|^2}{\|Y_{j+1}^{-1}(t_{j+1}(w))\|^2}.$$

Since 
$$\phi_j(w)=\frac{\|Y_j^{-1}(t_{j}(w))\|^2}{\|t_{j}(w)\|^2},\quad 1\leq j\leq n,$$ 
and the condition $(2)$ is satisfied, 
we then have
\begin{eqnarray*}
\theta_{j,j+1}(B)&=&\theta_{j,j+1}(\widetilde{T})\\
&=&\frac{\phi_j(w)}{\phi_{j+1}(w)}\theta_{j,j+1}(T)\\
&=&\frac{\|Y_j^{-1}(T_{j,j+1}(t_{j+1}(w)))\|^2}{\|T_{j,j+1}(t_{i+1}(w))\|^2}\frac{\|t_{j+1}(w)\|^2}{\|Y_{j+1}^{-1}(t_{j+1}(w))\|^2}\frac{\|T_{j,j+1}(t_{j+1}(w))\|^2}{\|t_{j+1}(w)\|^2}\\
&=&\frac{\|Y_j^{-1}T_{j,j+1}(t_{j+1}(w))\|^2}{\|Y_{j+1}^{-1}(t_{j+1}(w))\|^2}\\
&=&\theta_{j,j+1}(A).
\end{eqnarray*}
Hence there exists a diagonal unitary operator $V$ whose $(j,j)$-entry is denoted by $V_j$ such that $V^*BV=A$. 
Consider
\begin{eqnarray*}
VY^{-1}TYV^*
&=&V\overline{T} V^*\\
&=&\left ( \begin{smallmatrix}\widetilde{T}_{1,1}&\widetilde{T}_{1,2}& \mathsmaller{V_1Y_1^{-1}T_{1,3}Y_3V^*_3}&\mathsmaller{V_1Y_1^{-1}T_{1,4}Y_4 V^*_4}& \cdots&\mathsmaller{V_1Y_1^{-1}T_{1,n}Y_nV^*_n}\\
0&\widetilde{T}_{2,2}&\widetilde{T}_{2,3}&\mathsmaller{V_2X_2T_{2,4}X_4V^*_4}&\cdots&\mathsmaller{V_2Y_2^{-1}T_{2,n}Y_nV^*_n} \\
0&0&\widetilde{T}_{3,3}&\widetilde{T}_{3,4}&\cdots&\mathsmaller{V_3Y_3^{-1}T_{3,n}Y_nV^*_n} \\
\vdots&\vdots&\vdots&\ddots&\ddots&\vdots\\
0&0&\cdots&\widetilde{T}_{n-2,n-2}&\widetilde{T}_{n-2,n-1}&\mathsmaller{V_{n-2}Y_{n-2}^{-1}T_{n-2,n}Y_nV^*_n}\\
0&0&0&\cdots&\widetilde{T}_{n-1,n-1}&\widetilde{T}_{n-1,n}\\
0&0&0&\cdots&0&\widetilde{T}_{n,n}
\end{smallmatrix}\right ).
\end{eqnarray*}
By Lemma \ref{J21}, there exists $\widetilde{K}\in \m{K}(\m{H})$ such that  $I+\widetilde{K}$ is invertible and 
$$(I+\widetilde{K})VY^{-1}TYV^*(I+\widetilde{K})^{-1}=\widetilde{T}.$$
 Thus $T\sim_{\m{U+K}}\widetilde{T}$.

On the other hand, suppose that  $T\sim_{\m{U+K}}\widetilde{T}$.
Then there is a unitary  operator $U$ and compact operator $K$ such that $T(U+K)=(U+K)\widetilde{T}$. 
By Proposition \ref{pjjkm}, $U+K$ is upper triangular. 
By Lemma \ref{Jiangandji}, we can assume that $$U+K=\left ( \begin{smallmatrix}U_{1,1}+K_{1,1} & K_{1,2}& K_{1,3}& \cdots & K_{1,n}\\
0&U_{2,2}+K_{2,2}&K_{2,3}&\cdots&K_{2,n}\\
\vdots&\ddots&\ddots&\ddots&\vdots\\
0&\cdots&0&U_{n-1,n-1}+K_{n-1,n-1}&K_{n-1,n}\\
0&\cdots&\cdots&0&U_{n,n}+K_{n,n}
\end{smallmatrix}\right ),$$
where $U=\diag\{U_{1,1},U_{2,2},\cdots,U_{n,n}\}, K=(K_{\ell,j})_{n\times n}, K_{\ell,j}=0, \ell>j.$ 

By Lemma \ref{J3}, we have 
\begin{eqnarray*}
&&
\left ( \begin{smallmatrix}T_{1,1} & T_{1,2}& 0& \cdots & 0\\
0&T_{2,2}&T_{2,3}&\cdots&0\\
\vdots&\ddots&\ddots&\ddots&\vdots\\
0&\cdots&0&T_{n-1,n-1}&T_{n-1,n}\\
0&\cdots&\cdots&0&T_{n,n}
\end{smallmatrix}\right )
\left ( \begin{smallmatrix}U_{1,1} +K_{1,1}& 0& 0& \cdots & 0\\
0&U_{2,2}+K_{2,2}&0&\cdots&0\\
\vdots&\ddots&\ddots&\ddots&\vdots\\
0&\cdots&0&U_{n-1,n-1}+K_{n-1,n-1}&0\\
0&\cdots&\cdots&0&U_{n,n}+K_{n,n}
\end{smallmatrix}\right )
\\
&=&
\left ( \begin{smallmatrix}U_{1,1} +K_{1,1}& 0& 0& \cdots & 0\\
0&U_{2,2}+K_{2,2}&0&\cdots&0\\
\vdots&\ddots&\ddots&\ddots&\vdots\\
0&\cdots&0&U_{n-1,n-1}+K_{n-1,n-1}&0\\
0&\cdots&\cdots&0&U_{n,n}+K_{n,n}
\end{smallmatrix}\right )
\left ( \begin{smallmatrix}\widetilde{T}_{1,1} & \widetilde{T}_{1,2}&0& \cdots & 0\\
0&\widetilde{T}_{2,2}&\widetilde{T}_{2,3}&\cdots&0\\
\vdots&\ddots&\ddots&\ddots&\vdots\\
0&\cdots&0&\widetilde{T}_{n-1,n-1}&\widetilde{T}_{n-1,n}\\
0&\cdots&\cdots&0&\widetilde{T}_{n,n}
\end{smallmatrix}\right ).
\end{eqnarray*}
It follows that 
 $$T_{j,j}(U_{j,j}+K_{j,j})=(U_{j,j}+K_{j,j})\widetilde{T}_{j,j}, \quad (U_{j,j}+K_{j,j})T_{j,j+1}=\widetilde{T}_{j,j+1}(U_{j+1,j+1}+K_{j+1,j+1}).$$
 Furthermore, $U_{j,j}+K_{j,j}$ is invertible and $U_{j,j}$ is a unitary operator and $K_{j,j}$ is a compact operator. 
Let $$\phi_j(w):=\frac{\|(U_{j,j}+K_{j,j})^{-1}(t_{j}(w))\|^2}{\|t_{j}(w)\|^2},\quad 1\leq j\leq n.$$

Following  the same argument as in the sufficient part and from Proposition \ref{u+k}, we can conclude that  
$$K_{T_{j,j}}-K_{\widetilde{T}_{j,j}}=\frac{\partial^2}{\partial w\partial\overline{w}}\ln(\phi_j).$$
Thus, we have that 
\begin{align*}
C_{\lambda}(E_{\diag T})=&\det\left(I-i\frac{\lambda}{2\pi}\Phi\right)\\
=& \det\left[\left(I+\frac{i\lambda}{2\pi}K_{E_{diag\tilde{T}}}(w)\right) \oplus  {\mathcal A}(w)\right],
\end{align*}
where  $\Phi(w)=\displaystyle \bigoplus\limits_{j=1}^n\frac{\partial^2}{\partial w\partial\overline{w}}\ln\left(\phi_j(w)\|\tilde{t}_j(w)\|^2)\right)$, $\tilde{t}_j$ is a non zero section of bundle $E_{\widetilde{T}_{j,j}}$.
Furthermore, 
$$\frac{\phi_j}{\phi_{j+1}}\theta_{j,j+1}(T)=\frac{\|(U_j+K_j)T_{j,j+1}(t_{j+1})\|^2}{\|(U_{j+1}+K_{j+1})(t_{j+1})\|^2}=\frac{\|\widetilde{T}_{j,j+1}(U_{j+1}+K_{j+1})(t_{j+1})\|^2}{\|(U_{j+1}+K_{j+1})(t_{j+1})\|^2}=\theta_{j,j+1}(\widetilde{T}).$$
This finishes the proof of the necessary part. 
\end{proof}

\section{Unitary Invariants}
In this section, we completely determined the unitary invariants of this class $\m{NCFB}_n(\Omega)$ of essentially normal Cowen-Douglas operators.

\begin{thm}\label{thm:unitary}
Let $T,\tilde{T}\in \mathcal{NCFB}_n(\Omega)$. 
Suppose that $T_{\ell,j}=\phi_{\ell,j}(T_{\ell,i})T_{\ell,\ell+1}T_{\ell+1,\ell+2}\cdots T_{j-1,j}$ and $\tilde{T}_{\ell,j}=\tilde{\phi}_{\ell,j}(\tilde{T}_{\ell,\ell})\tilde{T}_{\ell,\ell+1}\tilde{T}_{\ell+1,\ell+2}\cdots \tilde{T}_{j-1,j}$  
Then
  $T\sim_{u}\widetilde{T}$ if and only if  the following statements hold:
 \begin{enumerate}
\item[(1)] $C_{\lambda}(E_{\diag{T}})=C_{\lambda}(E_{\diag{\tilde{T}}})$ ;
\item[(2)]
$\theta_{\ell,j}(T)=\theta_{\ell,j}(\tilde{T})$, 
\item[(3)] $\phi_{\ell,j}=\tilde{\phi}_{\ell,j}$, $1\leq \ell\leq j\leq n.$
 \end{enumerate} 
\end{thm}

\begin{proof}
Suppose $T\sim_u\tilde{T}$ in $ \mathcal{NCFB}_n(\Omega)$, i.e. there is a unitary operator $U$ such that $UT=\tilde{T}U$. By Theorem 3.5 in \cite{JJKM17}, such an intertwining unitary must be diagonal, that is,
$U=U_1\oplus \cdots \oplus U_{n},$ for some choice of $n$ unitary operators $U_1, \ldots , U_{n}.$

Since $U_\ell T_\ell=\tilde{T}_\ell U_\ell, 1\leq \ell \leq n,$ and
$U_\ell T_{\ell,j}=\tilde{T}_{\ell,j}U_{j}, 1\leq \ell< j\leq  n.$ 
In particular, $U_\ell T_{\ell,\ell+1}=\tilde{T}_{\ell,\ell+1}U_{\ell+1}, 1\leq \ell\leq n-1.$  
Suppose that there exists a holomorphic function $\phi_\ell$ such that 
$U_\ell(t_\ell)=\phi_\ell\tilde{t}_\ell$. 
Then we have that 
\begin{align*}
U_\ell T_{\ell,\ell+1}(t_{\ell+1})=&U_\ell(t_\ell)=\phi_\ell\tilde{t}_\ell\\
=&\tilde{T}_{\ell,\ell+1}U_{\ell+1}(t_{\ell+1})\\
=&\phi_{\ell+1}\tilde{t}_\ell.
\end{align*}
Thus, we have $\phi_\ell=\phi_{\ell+1}$. 
We then set $\phi:=\phi_\ell, \ell=1,2\cdots,n$. 
That means 
\begin{eqnarray}\label{inv1} 
U_\ell (t_\ell(w))=\phi(w)\tilde{t}_\ell(w),\; 1\leq \ell\leq n, w\in \Omega,
\end{eqnarray} 
where $\phi$ is some non zero holomorphic function.   
Thus
$$K_{T_{\ell}}=K_{\tilde{T}_{\ell}}\quad \mbox{ and }\quad  \frac{\|t_{\ell+1}(w)\|}{\|\tilde{t}_{\ell+1}(w)\|}=\frac{\|t_\ell(w)\|}{\|\tilde{t}_\ell(w)\|},\quad 1\leq \ell \leq n-1.$$
or 
$$K_{T_{\ell}}=K_{\tilde{T}_{\ell}}\quad \mbox{ and }\quad \theta_{\ell-1,\ell}(T)=\tilde{\theta}_{\ell-1,\ell}(\tilde{T})\; 1\leq \ell\leq n.$$ 
It follows that $C_{\lambda}(E_{\diag{T}})=C_{\lambda}(E_{\diag{\tilde{T}}})$. 

On the other hand, we have that $T_{\ell,j}(t_j)=\phi_{\ell,j}t_\ell$ and $ T_{\ell,j}(\tilde{t}_j)=\tilde{\phi}_{\ell,j}\tilde{t}_\ell.$ 
By the equations $U_\ell T_{\ell,j}=\tilde{T}_{\ell,j}U_{j}, 1\leq \ell<j\leq n-1,$ we have that 
$$\phi_{\ell,j}=\tilde{\phi}_{\ell,j}, \quad 1\leq \ell<j\leq n$$ 
Since $ \frac{\|t_{\ell+1}\|}{\|\tilde{t}_{\ell+1}\|}=\frac{\|t_\ell\|}{\|\tilde{t}_\ell\|},\;1\leq \ell\leq n-1,$ we also have that 
$$ \frac{\|t_{j}\|}{\|\tilde{t}_{j}\|}=\frac{\|t_\ell\|}{\|\tilde{t}_\ell\|},\;1\leq \ell<j\leq n.$$ 
Since $\phi_{\ell,j}=\tilde{\phi}_{\ell,j}$, then it follows that 
$$ \theta_{\ell,j}(T)=\frac{\|T_{\ell,j}(t_{j})\|}{\|t_{j}\|}=\frac{\|\phi_{\ell,j}t_{\ell}\|}{\|t_{j}\|}=\frac{\|\tilde{\phi}_{\ell,j}\tilde{t}_{\ell}\|}{\|\tilde{t}_{j}\|}=\frac{\|\tilde{T}_{\ell,j}(\tilde{t}_{j})\|}{\|\tilde{t}_{j}\|}=\tilde{\theta}_{\ell,j}(\tilde{T}),\;1\leq \ell<j\leq n.$$
This finishes the proof of the necessary part. 

Conversely assume that $T$ and $\tilde{T}$  are operators in $\mathcal{NCFB}_n(\Omega)$ for which these invariants are the same.  
Since $C_{\lambda}(E_{\diag{T}})=C_{\lambda}(E_{\diag{\tilde{T}}})$, then we have that 
$$\prod\limits_{k=1}^{n}\left(1+\lambda \frac{i}{2\pi}K_{T_{k,k}}\right)=\prod\limits_{k=1}^{n}\left(1+\lambda \frac{i}{2\pi}K_{\tilde{T}_{k,k}}\right).$$
Thus, there exists a permutation $\Xi$ on $\{1,2,\cdots,n\}$ such that 
$$K_{T_{\ell,\ell}}=K_{\tilde{T}_{\Xi(\ell),\Xi(\ell)}},\quad  1\leq \ell\leq n.$$
That means there exist unitary operators  $U_\ell$ such that 
$$U_\ell \tilde{T}_{\Xi(\ell),\Xi(\ell)}=T_\ell U_\ell, \quad  1\leq \ell\leq n. $$

Now we will prove that $\Xi(\ell)=\ell, 1\leq \ell \leq n$. 
Otherwise, without loss of generality, we can assume that $\Xi(n)<n$. 
In fact, if $\Xi(n)=n$, we can assume $\Xi(n-1)<n-1.$  Since $U_n\tilde{T}_{\Xi(n),\Xi(n)}=T_{n,n}U_n,$ 
then we have that 
\begin{align*}
U_n\tilde{T}_{\Xi(n),\Xi(n)}\tilde{T}_{\Xi(n),n}=&T_nU_n\tilde{T}_{\Xi(n),n}\\
U_n\tilde{T}_{\Xi(n),n}\tilde{T}_{n,n}=&T_{n,n}U_n\tilde{T}_{\Xi(n),n}.
\end{align*}
Notice that $$U_{\Xi^{-1}(n)}\tilde{T}_{n,n}=T_{\Xi^{-1}(n),\Xi^{-1}(n)}U_{\Xi^{-1}(n)},$$
and $\Xi^{-1}(n)<n$ (Recall $\Xi(n)<n$). Then we have that 
$$
U_n\tilde{T}_{\Xi(n),n}\tilde{T}_{n,n}U_{\Xi^{-1}(n)}^{*}=T_{n,n}U_n\tilde{T}_{\Xi(n),n}U_{\Xi^{-1}(n)}^{*}.
$$
It follows that $$
U_n\tilde{T}_{\Xi(n),n}U_{\Xi^{-1}(n)}^{*}T_{\Xi^{-1}(n),\Xi^{-1}(n)}=T_{n,n}U_n\tilde{T}_{\Xi(n),n}U_{\Xi^{-1}(n)}^{*}.
$$
Thus, $U_n\tilde{T}_{\Xi(n),n}U_{\Xi^{-1}(n)}^{*}\in \Ker\tau_{T_{n,n}, T_{\Xi^{-1}(n),\Xi^{-1}(n)}}$. 
Since  $\Xi^{-1}(n)<n$, by Lemma \ref{order}, we know that $ \Ker \tau_{T_{n,n}, T_{\Xi^{-1}(n),\Xi^{-1}(n)}}=\{0\}$. 
Thus we have that $\tilde{T}_{\Xi(n),n}=0$. 
This is a contradiction. 
So we have that $\Xi(\ell)=\ell, 1\leq \ell\leq n$, and 
$$T_{\ell,\ell}=U_\ell \tilde{T}_{\ell,\ell} U_\ell^{*}, \quad K_{T_{\ell,\ell}}=K_{\tilde{T}_{\ell,\ell}} \quad  \ell=1,2,\ldots,n.$$

Equality  of the two curvatures $K_{T_{\ell}}=K_{\tilde{T}_{\ell}}$ together with the equality of the second fundamental forms $\frac{\|t_{\ell+1}\|}{\|\tilde{t}_{\ell+1}\|}=\frac{\|t_\ell\|}{\|\tilde{t}_\ell\|},\;1\leq \ell \leq n-1$ implies that there exist a non-zero holomorphic function $\phi$ defined on $\Omega$ (if necessary,  one may choose a domain $\Omega_0\subseteq \Omega$ such that $\phi$ is non zero on $\Omega_0$) such that
$$\|t_\ell (w)\|=|\phi(w)|\,\|\tilde{t}_\ell(w)\|, \quad  1 \leq \ell \leq n.$$
For $ 1 \leq \ell \leq n,$ we define $U_\ell:\mathcal{H}_\ell \to \widetilde{\mathcal{H}}_\ell $ by the formula 
$$U_\ell (t_\ell(w))=\phi(w)\tilde{t}_\ell(w),\; w\in\Omega.$$
and extend to the linear span of these vectors. 
For $1 \leq \ell \leq n,$
\begin{eqnarray*}
\|U_\ell (t_i(w))\|&=&\|\phi(w)\tilde{t}_\ell (w)\|\\
&=& |\phi(w)|\|\tilde{t}_\ell(w)\|\\
&=&\|t_\ell(w)\|.
\end{eqnarray*}
Thus each $U_\ell $ extends to an isometry from $\mathcal{H}_\ell$ to $\widetilde{\mathcal{H}}_\ell.$ 
Since $U_\ell$ is isometric and $U_\ell T_\ell=\tilde{T}_\ell U_\ell.$ 
Let $\Omega_0$ be arbitrary open connected subset of $\Omega$. 
Since  $\bigvee\limits_{w\in \Omega_0}\{\tilde{t}_\ell (w)\}=\mathcal{H}_\ell$,  then we can see  each $U_\ell$ is unitary. 

By a direct computation, we can  see that $U_\ell T_{\ell,\ell+1}=\tilde{T}_{\ell,\ell+1}U_{\ell+1}$ for $1\leq \ell \leq n-1$ also.
For any  $w\in \Omega$ and $0\leq \ell<j\leq n-2$ with $j-\ell \geq 2$, we have
$$ U_\ell (T_{\ell,j}(t_j(w)))=U_\ell(\phi_{\ell,j}(w)t_\ell(w))=\phi_{\ell,j}(w)\phi(w)\tilde{t}_\ell(w)=\tilde{\phi}_{\ell,j}(w)\phi(w)\tilde{t}_\ell(w)=\tilde{T}_{\ell,j}(U_j({t}_j(w)))$$
which implies that $$U_\ell T_{\ell,j}=\tilde{T}_{\ell,j}U_j.$$
Hence setting $U= U_1 \oplus \cdots \oplus U_{n},$ we see that $U$ is unitary and $UT=\tilde{T}U$.
This completes the proof of the Theorem.
\end{proof}

\section{Summary and Further Research}

In Section 3, it is critical to find an operator in $\m{NCFB}_n(\Omega)$ (Theorem \ref{example}) to show that $\m{NCFB}_n(\Omega)$ contains many nontrivial operators when $\Omega$ is simply connected.
It is natural to ask the following questions:
\begin{itemize}
\item[1.] Can we construct an operator in $\m{NCFB}_n(\Omega)$ when $\Omega$ is connected?
\item[2.] Can we generalize Proposition \ref{JWL} to the case $\Omega$ is (2-)connected?
\end{itemize}

In 1984, Misra \cite{Mis84} introduced homogeneous operators.
An operator $T\in \m{L}(\m{H})$ is homogeneous if $\sigma(T)\subset \overline{\mathbb{D}}$ and $\phi_\alpha(T)$ is unitarily equivalent to $T$ for all M\''{o}bius transformation $\phi_\alpha$;
Misra also introduce homogeneous for similarity.
An operator $T\in \m{L}(\m{H})$ is weakly homogeneous if $\sigma(T)\subset \overline{\mathbb{D}}$ and $\phi_\alpha(T)$ is similar to $T$ for all M\''{o}bius transformation $\phi_\alpha$;
In the framework of Herrero, we introduce ($\m{U+K}$)-homogeneous between homogeneous and weakly homogeneous.
\begin{defn}
An operator $T\in \m{L}(\m{H})$ with $\|T\|\leq 1$ is ($\m{U+K}$)-homogenous if there exists a unitary operator $U_{\alpha}$ and a compact operator $K_{\alpha}$ such that $U_{\alpha}+K_{\alpha}$ is invertible and 
$$(U_{\alpha}+K_{\alpha})T=\phi_{\alpha}(T)(U_{\alpha}+K_{\alpha}), \alpha\in \mathbb{D}. $$
for all M\"{o}bius transformation $\phi_\alpha$ of the unit disk.
\end{defn}

In \cite{Mis84}, Misra showed that an operator $T\in \m{B}_1(\Omega)$ with $\|T\|\leq 1$ is homogeneous if and only if its curvature is $-\lambda(1-|w|^2)^{-2}$ for some positive number $\lambda$.

\begin{qn}
Can we give a simple characterization of ($\m{U+K}$)-homogeneous operators?
\end{qn}

\end{document}